\documentclass[11pt]{article}
\usepackage{amsmath}
\usepackage{amssymb,enumerate}
\usepackage{amsfonts}
\usepackage{amsthm}
\usepackage{ascmac}
\usepackage{latexsym}
\usepackage{mathrsfs}
\usepackage{amsrefs}

\usepackage{cases}
\DeclareFontFamily{U}{mathc}{}
\DeclareFontShape{U}{mathc}{m}{it}%
{<->s*[1.03] mathc10}{}

\DeclareMathAlphabet{\mathscr}{U}{mathc}{m}{it}

\usepackage{enumerate}
\usepackage[numbers]{natbib}
\usepackage{mathscinet}
\usepackage{setspace}

\usepackage[dvipdfmx]{graphicx}

\usepackage{hyperref}
\hypersetup{colorlinks=true,linkcolor=blue,
citecolor=blue}

\usepackage{floatflt}

\usepackage{mathptmx}

\usepackage[scaled]{helvet}
\usepackage{fancybox,ascmac}

\usepackage{bbm}

\usepackage[inline]{enumitem}

\usepackage{latexsym}
\newsavebox{\toy}
\savebox{\toy}{\framebox[0.65em]{\rule{0cm}{1ex}}}
\newcommand{\QED}{\usebox{\toy}\end{demo}}

\theoremstyle{theorem}
\newtheorem{theorem}{Theorem}[section]
\newtheorem*{theorem*}{Theorem}
\newtheorem{definition}[theorem]{Definition}
\newtheorem{prop}[theorem]{Proposition}
\newtheorem{lemma}[theorem]{Lemma}

\newtheorem{cor}[theorem]{Corollary}
\newtheorem{rem}[theorem]{Remark}
\newtheorem{conjecture}[theorem]{Conjecture}

\makeatletter
             \@addtoreset{equation}{section}
\makeatother

\newcommand{\dis}{\displaystyle}

\usepackage{comment}

\newcommand{\N}{{\mathbb{N}}}

\newcommand{\R}{{\mathbb{R}}}

\newcommand{{\rd}}{\R^d}

\newcommand{\IP}{{\mathbb P}}
\newcommand{\IE}{\mathbb E}

\newcommand{\DP}{{\mathrm P}}
\newcommand{\DE}{{\mathrm E}}

\renewcommand{\b}{\beta}

\newcommand{\e}{\varepsilon}

\newcommand{\dd}{\text{\rm d}}             

\newcommand{\bsm}{\begin{smallmatrix}}
\newcommand{\esm}{\end{smallmatrix}}

\newcommand{\MN}[1]{\textcolor{red}{#1}}

\begin{document}

\author{Makoto Nakashima\thanks{nakamako@math.nagoya-u.ac.jp, Graduate School of Mathematics, Nagoya University, Furocho, Chikusaku, Nagoya, Japan}}
\title{Feynman--Kac formula for the heat equation with a one-center point interaction in $d=3$.}
\date{}

 \maketitle


\begin{abstract}
We study Schr\"odinger operators with a one-center point interaction, formally defined by
\begin{align*}
-\Delta_\alpha=-\Delta+\alpha\,\delta_0(\cdot),
\end{align*}
for $\alpha\in\R$, and the associated heat equation
\begin{align}
\partial_t u=\tfrac{1}{2}\Delta_{\alpha} u,\quad u(0,x)=u_0(x)\in C_c^{\infty}(\R^3\setminus\{0\}).\label{eq:HEapp}
\end{align}
Here $\Delta$ denotes the Laplacian (self-adjoint on $L^2(\R^3)$) and $\delta_x$ the Dirac measure at $x$. The operator $-\Delta_\alpha$ can be realized either as a self-adjoint extension of $-\Delta|_{C_0^{\infty}(\R^3\setminus\{0\})}$ in $L^2(\R^3)$, or as the norm-resolvent limit of $-\Delta+\lambda_\e V(\cdot/\e)$ for suitable $\lambda_\e$ and $V:\R^3\to\R$.

In this paper we construct, for each $t>0$ and $x\in\R^3\setminus\{0\}$, a probability law on path space and a normalizing function $G_t^\alpha(x)$ giving the following probabilistic representation of the solution to \eqref{eq:HEapp}:
\begin{align*}
u(t,x)=G_t^\alpha(x)\,\mathbb{E}\bigl[u_0\bigl(W^{t,x}(t)\bigr)\bigr],
\end{align*}
where $\{W^{t,x}(s):0\le s\le t\}$ is a continuous process depending on $(t,x,\alpha)$. The result provides a Feynman--Kac type formula for the heat equation with a one-point interaction in three dimensions.
\end{abstract}



We denote by $(\Omega,{\cal F},P)$ a probability space and write $P[X]$ for the expectation of a random variable $X$.  
Let $C_{x_1,\cdots,x_p}$ or $C(x_1,\cdots,x_p)$ be a non-random constant which depends only on the parameters $x_1,\cdots,x_p$. 


\section{Introduction}

We consider  Schr\"odinger operators formally defined by \begin{align}
-\Delta+\lambda \delta_0.\label{eq:formdeltapot}
\end{align}
This model has been studied in nuclear physics, solid-state physics, and electromagnetic theory \cite{KP31,BP35,Tho35}.

Berezin and Faddeev gave the first rigorous analysis of \eqref{eq:formdeltapot} for $d=3$ \cite{BF61}, describing it as a family of self-adjoint extensions of the nonnegative operator \begin{align*}
{\mathrm{H}}_0:=-\Delta|_{C_0^{\infty}(\R^d-\{0\})}
\end{align*}
in $L^2(\R^d)$ for $d=3$. Krein's theory shows that these self-adjoint extensions form the one-parameter family $\{-\Delta_\alpha\}_{\alpha\in(-\infty,\infty]}$ \cite[I.~Theorem 1.1.1.]{AGHKH12}. Also, $-\Delta_\alpha$ could be approximated by scaled short-range interactions in the sense  of norm resolvent convergence \cite[I.~Theorem 1.2.5]{AGHKH12}, \cite{Fri72,AHK81,AGHK82}: Let $\widetilde{V}:\R^3\to \R$ be a measurable function satisfying $\int_{\R^3\times \R^3}\frac{|\widetilde{V}(x)||\widetilde{V}(x')|}{|x-x'|^2}\dd x\dd x'<\infty$. Then, the operators $-\Delta+\frac{\lambda(\e)}{\e^2}\widetilde{V}\left(\frac{\cdot}{\e}\right)$ converges to $-\Delta_\alpha$ in the norm resolvent convergence sense if $\widetilde{V}$ satisfies some nice conditions, where $\widetilde{V}^\e(x)=\frac{1}{\e^2}\widetilde{V}\left(\frac{x}{\e}\right)$ and $\lambda(\e)\sim -1-\gamma \e$  for $\gamma\in \R$ is the coupling constant.

For the two-dimensional case, the same result holds after replacing the coupling constant by $\lambda(\e)\sim c\left(\frac{1}{-\log \e}+\frac{\lambda}{(-\log \e)^2}\right)$ for a suitable constant $c$ \cite[I.5]{AGHKH12}. 

For the one-dimensional case, $\mathrm{H}_0$ has a four-parameter family of self-adjoint extensions in $L^2(\R)$ parametrized by the group $U(2)$ of $2\times 2$ unitary matrices \cite{ABD95}. For higher dimensions, self-adjoint extensions of $\mathrm{H}_0$ are trivial \cite[Theorem X.11]{RS75}. The reader may refer to \cite{AGHKH12} and the references therein, and \cite{AF18} is a recent survey.

The fundamental solutions of heat equations and Schr\"odinger equations with one-point interaction have been discussed in \cite{ABD95} and are explicitly given for each $d=1,2,3$: For instance, the fundamental solutions of heat equation with one-point interaction \begin{align}
\partial_t u=\Delta_\alpha u\label{eq:HEonepoint}
\end{align}
is given by \begin{align}
P_t^\alpha(x,y)=\frac{2t}{|x||y|}P(t;|x|+|y|)-\frac{8\pi \alpha t}{|x||y|}\int_0^\infty e^{-4\pi \alpha u}P(t;u+|x|+|y|)\dd u\label{eq:FunSol}
\end{align}
for $x,y\not=0$ and $t>0$, where $P(t;R)=\frac{1}{(4\pi t)^\frac{3}{2}}\exp\left(-\frac{R^2}{4t}\right)$ for $t>0$ and $R>0$.  

In \cite{CKMV10}, Cranston et.al.~consider the approximations of $\frac{1}{2}\Delta_\alpha$ by \begin{align}
\frac{1}{2}\Delta +\left(\frac{\pi^2}{8}+\gamma \e\right)\frac{1}{\e^2}V\left(\frac{\cdot}{\e}\right),\quad \int_{\R^3}|V(x)|\dd x=\frac{4\pi}{3},\label{eq:HEapproxCKMV}
\end{align}
and the corresponding Gibbs measures on Wiener space $(C([0,T],\R^3),\mathcal{B}(C([0,T],\R^3)))$ given by \begin{align}
\dd P_{\gamma,T}^{x,\e}(\omega)=\frac{1}{Z_{\gamma,T}^\e(x)}\exp\left( \left(\frac{\pi^2}{8}+\gamma \e\right)\int_0^T \frac{1}{\e^2}V\left(\frac{\omega_s}{\e}\right)\dd s\right)\dd P^x(\omega),\label{eq:FKHamil}
\end{align}
where $V$ is a bounded and compactly supported function, $Z_{\gamma,T}^\e(x)$ is the partition function, and $P^x$ denotes the law of three-dimensional Brownian motion started at $x$. Then they proved the weak convergence of $\{P_{\gamma,T}^{x,\e}\}$ to $\overline{P}_{\gamma,T}^x$ as $\e\to 0$ \cite[Theorem 2.3]{CKMV10}. In their proof, they  show that the convergence of the solutions of heat equations associated with \eqref{eq:HEapproxCKMV} to the solutions of the heat equation with one-center point interaction \eqref{eq:HEonepoint}. 

Our goal in this paper is to construct the path $\omega$ under $\overline{P}_{\gamma,T}^x$ explicitly. The construction of $\overline{P}^x_{\gamma,T}$ in \cite{CKMV10} uses the fundamental solution \eqref{eq:FunSol} of the heat equation with one-point interaction, which makes it difficult to infer pathwise properties of the process under $\overline{P}_{\gamma,T}^x$.

\begin{rem}
For the two-dimensional case, Chen gives a Feynman-Kac representation of \eqref{eq:HEal} in \cite{Che25}, where two proofs are provided. He approximates $\overline{P}_{\gamma,T}^x$ using transformations of Bessel processes and in particular studies excursions of these processes. Caravenna et al.~study the time-space discrete counterpart in the disordered model literature \cite{CSZ19}. Cranston and Molchanov study the continuous-time, discrete-space counterpart for higher dimensions $d\geq 3$ in \cite{CM19}. Renormalized partition functions such as \eqref{eq:FKHamil} also appear in the study of directed polymers in random environment and multiplicative stochastic heat equations related to high-dimensional KPZ equations \cite{CSZ20,CSZ23,Cha23,CD20,CNN22,DGRZ20,DGRZ21,Gu20,MU18,NN23}. 

For the one-dimensional case, two types of interactions are known: the $\delta$-interaction and the $\delta'$-interaction, where $\delta'$ is the derivative of the Dirac measure $\delta$. The approximation of $\delta$-interaction is studied in \cite{AGHKK84,BG69,DG85,Fri72,Zor80}, and the approximation of $\delta'$-interaction is studied in \cite{AN00,CH93,GK85,Seb86}. There are many papers on the one-dimensional case, and the reader may refer to \cite[Appendix K.1]{AGHKH12}. The author is not aware of results on a Feynman-Kac representation for the heat equation with a one-center point interaction in one dimension.

\end{rem}


\subsection{The zero-energy resonance}
In this paper, we fix the potential function $\widetilde{V}$ given in \eqref{eq:Vdef} below. This $\widetilde{V}$ is suitable to approximate $-\Delta_\alpha$:  \begin{align*}
\mathrm{H}_{\e}:=-\Delta+\frac{\lambda(\e)}{\e^2}\widetilde{V}\left(\frac{\cdot}{\e}\right)
\end{align*}
for $\lambda(0)=1$ and $\lambda'(0)\in \R$ approximates $-\Delta_\alpha$ in norm convergence sense \cite[I.~Theorem 1.2.4]{AGHKH12} or \cite{AGHK82}.

To see this, we give a quick review of the approximation in \cite{AGHKH12}.

 The potential kernel (or the Green function) in $\R^3$ is given by \begin{align*}
\widetilde{G}(x,y)&=\frac{1}{4\pi |x-y|}=\Delta^{-1}(x,y)\\
&=\int_{0}^\infty p_{2t}(x,y)\dd t\\
&=\frac{1}{2}\int_0^\infty p_t(x,y)\dd t,\quad x,y\in\R^3, x\not=y,
\end{align*}
where $p_t(x,y)=p_t(x-y)=\frac{1}{(2\pi t)^{\frac{3}{2}}}\exp\left(-\frac{|x-y|^2}{2t}\right)$ is a Gaussian  density function.

Let $\widetilde{V}$ be an $L^1(\R^3)$-function belonging to Rollnik class, i.e.\begin{align*}
\int_{\R^3}\dd x\int_{\R^3}\dd y \frac{|\widetilde{V}(x)||\widetilde{V}(y)|}{|x-y|^2}<\infty.
\end{align*} 
Then, we set \begin{align*}
v(x)=|\widetilde{V}(x)|^\frac{1}{2},\quad u(x)=|\widetilde{V}(x)|^\frac{1}{2}\mathrm{sgn}[\widetilde{V}(x)].
\end{align*}  
Suppose that there exists an $L^2(\R^3)$ function $\phi$ such that \begin{align*}
\int_{\R^3}u(x)\widetilde{G}(x,y)v(y)\phi(y)\dd y=-\phi(x).
\end{align*}
Then, we define \begin{align}
\psi(x)=\int_{\R^3}\widetilde{G}(x,y)v(y)\phi(y)\dd y.\label{eq:resonance}
\end{align}
If $\psi\in L^2_{\mathrm{loc}}(\R^3)$ but $\psi\not\in L^2(\R^3)$, $\psi$ and the spectral point $0$ are  called a zero-energy resonance function and a resonance of $\mathrm{H}_1$, respectively. 

Theorem 2.1 in \cite{AGHK82} tells us that if $\mathrm{H}_1$ has a zero-energy resonance function and $(1+|x|)^2\widetilde{V}\in L^1(\R^3)$ (resp. $(1+|x|)^4\widetilde{V}(x)\in L^1(\R^3)$), then $\mathrm{H}_\e$ converges in norm resolvent (resp.~strong resolvent) sense to $-\Delta_\alpha$ for suitable $\lambda(\e)$.

\subsubsection{Choice of $\widetilde{V}$}

We fix  a function $\widetilde{V}$ for which $\mathrm{H}_1$ has a resonance and a resonance function. 

We define \begin{align}
H(a,x)=\int_a^\infty p_t(x)\dd t=\frac{1}{2(\pi)^{\frac{3}{2}} |x|}\int_0^{\frac{|x|^2}{2a}}u^{-\frac{1}{2}}e^{-u}\dd u\label{eq:Hdef}
\end{align}
for $a\geq 0$ and $x\in\R^3$. 
Then, we define a  spherical symmetric  potential $\widetilde{V}$ by \begin{align}
\widetilde{V}(x)=-\frac{2}{(2\pi )^{\frac{3}{2}}H(1,x)}\exp\left(-\frac{|x|^2}{2}\right)=-\frac{2}{H(1,x)}p_{1}(x).\label{eq:Vdef}
\end{align}


It is easy to see that \begin{align}
&\int_{\R^3}\widetilde{G}(x,y)\widetilde{V}(y)H(1,y)\dd y=-2\int_{0}^{\infty}\int_{\R^3}p_{2t}(x,y)p_1(y)\dd y\dd t=-H(1,x).\label{eq:ResonEq}
\end{align}
Thus, if we set \begin{align*}
\phi(x)=\sqrt{\widetilde{V}(x)}H(1,x)\in L^2(\R^3),
\end{align*}
then $\phi$ is an eigenfunction of the Hilbert-Schmidt operator $u(x)\widetilde{G}(x,\cdot)v(x):L^2(\R^3)\to L^2(\R^3)$ with eigenvalue $-1$. 

\begin{prop}
Let $\widetilde{V}$ be given by \eqref{eq:Vdef}. Then, $\psi$ defined by \eqref{eq:resonance} is a zero-energy resonance function of $\mathrm{H}_1$.
\end{prop}
\begin{proof}
We know from  \eqref{eq:ResonEq} that $\psi(x)=H(1,x)$. Also, it is easy to see from \eqref{eq:Hdef} that \begin{align*}
\lim_{|x|\to \infty}|x|H(1,x)=\frac{1}{2(\pi)^\frac{3}{2}}\Gamma\left(\frac{1}{2}\right)=\frac{1}{2\pi},
\end{align*}
where $\Gamma(s)=\int_0^\infty x^{s-1}e^{-x}\dd x$ is the Gamma function. Therefore, we find that $\psi\not \in L^2(\R^3)$. 

Since $\lim_{|x|\to 0}H(1,x)=\frac{1}{2^\frac{1}{2}\pi^\frac{3}{2}}$, $\psi\in L^2_\mathrm{loc}(\R^3)$.
\end{proof}

Thus $\widetilde{V}$ satisfies the assumptions of \cite[Theorem 2.1]{AGHK82}, and hence $-\Delta+\frac{\lambda(\e)}{\e^2}\widetilde{V}\left(\frac{\cdot}{\e}\right)$ converges to $-\Delta_\alpha$ in strong resolvent sense as $\e\to 0$ for some $\alpha\in(-\infty,\infty)$.

\hskip\baselineskip

To prepare for the Feynman--Kac formula, it is convenient to set $G(x,y)=G(x-y)=2\widetilde{G}(x,y)$ for $x,y\in\R^3$ and $V=\frac{\widetilde{V}}{2}$.

The following properties of $H$ are easily derived from the definition: \begin{align}
&\lim_{|x|\to \infty }|x|H(a,x)=\frac{1}{2\pi}\label{eq:Hprop1}\\
&H(a,0):=\lim_{x\to 0}H(a,x)=\pi^{-\frac{3}{2}}(2a)^{-\frac{1}{2}}\label{eq:Hprop2}\\
&\lim_{a\to 0}H(a,x)=G(x)\label{eq:Hprop3}\\
&H(a,x)=\frac{1}{\sqrt{a}}H\left(1,\frac{x}{\sqrt{a}}\right) \label{eq:Hprop4}.
\end{align}
In particular,  $H(a,x)$ is decreasing in $|x|$, and hence
we have
\begin{align}
H(a,x)\asymp \frac{C}{\max\{\sqrt{a},|x|\}}\label{eq:Hprop5}
\end{align}
for $a>0$ and $x\in \R^3$.

It is easy to see from the scaling property of Gaussian density functions that 
for $\e>0$, \begin{align*}
V^\e(x):=\frac{1}{\e^2}V\left(\frac{x}{\e}\right)=-\frac{1}{(2\pi \e^2)^{\frac{3}{2}}h^\e(x)}\exp\left(-\frac{|x|^2}{2\e^2}\right)=-\frac{1}{h^\e(x)}p_{\e^2}(x),
\end{align*}
where \begin{align*}
h^\e(x):=\int_{\e^2}^\infty p_{t}(x)\dd t=\frac{1}{2(\pi)^{\frac{3}{2}} |x|}\int_0^{\frac{|x|^2}{2\e^2}}u^{-\frac{1}{2}}e^{-u}\dd u=H(\e^2,x).
\end{align*}
Also, we write $h(x)=h^1(x)=H(1,x)$.

Then, $h^\e$ and $V^\e$ have the following properties. 
\begin{enumerate}[label=(P-\arabic*)]
\item\label{item:p1} (Asymptotics) \begin{align*}
\lim_{|x|\to 0}h^\e(x)=\pi^{-\frac{3}{2}}(2\e^2)^{-\frac{1}{2}}=:h^\e(0),\quad \lim_{|x|\to\infty}|x|h^\e(x)=\frac{1}{2\pi}
\end{align*}
and hence, $h^\e\in L_{\mathrm{loc}}^2(\R^3)$ but $h^\e\not\in L^2(\R^3)$.
\item\label{item:p2} (Convolution) For $x\in \R^3$ \begin{align*}
\int_{\R^3}G(x,y)V^\e(y)h^\e(y)\dd y=h^\e(x).
\end{align*}
\item\label{item:p3} (Scaling) We have \begin{align*}
V^\e(x)=\frac{1}{\e^2}V\left(\frac{x}{\e}\right),\quad h^\e(x)=\frac{1}{\e}h\left(\frac{x}{\e}\right)\text{for $x\in \R^3$.}
\end{align*} 
In particular, \begin{align*}
h^\e(x)\to \frac{1}{2\pi|x|}:=h^0(x)=G(x),\quad x\not=0,\quad \text{as }\e\to 0.
\end{align*}

\end{enumerate}

\subsection{Main result}

We consider the heat equations \begin{align}
\partial_t u^\e(t,x)=\frac{1}{2}\Delta u^\e(t,x)-\lambda(\e) V^\e(x)u^\e(t,x) \label{eq:HE}
\end{align}
for $ u(0,x)=u_0(x)\in C^2_c(\R^3\backslash \{0\})$, where we retake $\lambda(\e)=1+\gamma \e$ for simplicity of notations. Then, a Feynman-Kac formula leads to
\begin{align*}
u^\e(t,x)&=\DE^x\left[u_0(B_t)\exp\left(-\lambda(\e)\int_0^t V^\e(B_s)\dd s\right)\right],
\end{align*}
where we denote by $\DP^x$ and $\DE^x$ the probability measure of Brownian motion started at $x$ and the expectation w.r.t.~$\DP^x$.

Our choices of potential $V$ and coupling $\lambda(\e)$ differ from those in \cite{CKMV10}. Nevertheless, their convergence argument can be adapted to our setting, so $u^\e$ converges to $u_\alpha$, where $u_\alpha$ is the solution of the associated heat equation
\begin{align}
\partial_t u_\alpha(t,x)=\frac{1}{2}\Delta_\alpha u_\alpha (t,x).\label{eq:HEal}
\end{align}

Our main result gives a Feynman--Kac formula for \eqref{eq:HEal}.

\begin{theorem}\label{thm:Main}
For any $x\in \R^3\backslash\{0\}$ and $t>0$, there exist a function $G_t^\alpha(x)$ and a probability measure $\mu_{t,x}^\alpha $  on the Wiener space $\left(\Omega=C([0,\infty),\R^3), \mathcal{B}(C([0,\infty),\R^3))\right)$ such that  \begin{align*}
u_\alpha (t,x)=G_{t}^\alpha(x)\mu_{t,x}^\alpha\left[u_0(\omega^{t,x}(t))\right]
\end{align*}
for $u_0\in C_b(\R^3)$.
\end{theorem}

In the following, we will see $G^\alpha_{t}(x)$ and $\mu_{t,x}^\alpha$. 

We define the Green function as the function which satisfies \begin{align}
\int_{\R^3}G_{\e,t}(x,y)f(y)\dd y=\DE^x\left[\exp\left(-\lambda(\e)\int_0^t V^\e(B_s)\dd s \right)f(B_t)\right]\label{eq:FKsta}
\end{align}
and we set \begin{align*}
G_{\e,t}(x)=\DE^x\left[\exp\left(-\lambda(\e)\int_0^t V^\e(B_s)\dd s \right)\right].
\end{align*}

Then, we focus on a Gibbs measure on the Wiener space $\left(\Omega=C([0,\infty),\R^3), \mathcal{B}(C([0,\infty),\R^3))\right)$ by \begin{align}
\mu_{\e,t,x}(\dd \omega)=\frac{1}{G_{\e,t}(x)}\exp\left(-\lambda(\e)\int_0^t V^\e(\omega_s)\dd s \right)\mu(\dd \omega),\label{eq:Gibbs}
\end{align}
where $\mu$ is the Wiener measure on the Wiener space.

The proof of Theorem \ref{thm:Main} is completed if we show the convergence of $G_{\e,t}(x)$ and the weak convergence of $\mu_{\e,t,x}$. 

\begin{rem}
Our proof is purely an analysis of $G_{\e,t}(x)$ and $\mu_{\e,t,x}$, which are determined by $V^\e$. Therefore our analysis should extend to the countably many point interactions discussed in \cite[Chapter III]{AGHKH12}.
\end{rem}

\begin{rem} $\mu_{\e,t,x}$ can be regarded as the law of Brownian motion in a deterministic obstacle. Sznitman's textbook (which treats random obstacles) gives an SDE description of the dynamics of $\omega$ under $\mu_{\e,t,x}$ \cite[Chapter 6, (1.8)]{Szn98}:    
Let $\{\omega_s\}_{0\leq s\leq t}$ be the canonical path under $\mu_{\e,t,x}$. Then $\omega$ is the solution of the stochastic differential equation
\begin{align*}
&\dd\omega_s=\dd B_s+\frac{\nabla G_{\e,t-s}(\omega_s)}{G_{\e,t-s}(\omega_s)}\dd s,\quad 0\leq s\leq t,\\
&\omega_0=x,
\end{align*}
where $\{B_s\}_{0\leq s\leq t}$ is a standard Brownian motion in $\R^3$. 

We can see that $G_{\e,t}(x)$ and $\nabla G_{\e,t}(x)$ converge to $G_t^\alpha(x)$ and $\nabla G_{t}^\alpha(x)$ as $\e\to 0$ for $x\not=0$. Thus we may expect the limit process to satisfy the stochastic differential equation \begin{align}
&\dd\omega^{t,x}_s=\dd B_s+\frac{\nabla G_{t-s}^\alpha(\omega^{t,x}_s)}{G_{t-s}^\alpha(\omega^{t,x}_s)}\dd s\quad 0\leq s\leq t\label{eq:SDE1}\\
&\omega_0=x.\notag
\end{align}
This means the process under $\overline{P}_{\gamma,T}^x$ can be regarded as a Brownian motion with a time-dependent \textit{singular} drift $\frac{\nabla G_{t-s}^\alpha(\cdot)}{G_{t-s}^\alpha(\cdot)}$. 
Such a process has been studied in the literature of the stochastic differential equations, and the existence and uniqueness of the solution of \eqref{eq:SDE1} is known under some conditions on the drift. 
For instance, in \cite{BC05} the drift term belongs to some Kato class measure, and in \cite{KR05}, the time dependent drift term satisfies $L_{q-}L_p$-integrability.
However, we will see that $\frac{\nabla G_{t}^\alpha(x)}{G_t^\alpha(x)}$ is singular at the origin and they does not satisfies these assumptions.
So, the existence and uniqueness of the solution of \eqref{eq:SDE1} is not covered by the existing results, and it still remains open.
Mian discussed \eqref{eq:SDE1} in \cite{Mia26}.

In the two-dimensional case, Chen \cite{Che25} gives an SDE description of the process under $\overline{P}_{\gamma,T}^x$. 
He approximates $\overline{P}_{\gamma,T}^x$ using transformations of Bessel processes, focusing in particular on excursions. 
Recently Clark and Mian \cite{CM25} studied the local time of the process at the origin under $\overline{P}_{\gamma,T}^x$.


\end{rem}

\subsection{Interpretation of $\mu_{\e,t,x}$ via time-space Markov process.}

In this subsection, we will interpret the process $\omega$ under $\mu_{t,\e,x}$ via another process $\left\{W^{\e,t,x}(s):s\geq 0\right\}$.

First, we look at the expansion  \begin{align*}
&G_{\e,t}(x,y)\\
&=p_t(x,y)\\
&+\sum_{n\geq 1}(-\lambda(\e))^n\int_{\Lambda_n([0,t])}\int_{\R^{3n}}p_{s_1}\left(x,x_1\right)V^\e(x_1)\prod_{k=1}^{n-1}\left(p_{s_{k+1}-s_k}\left(x_k,x_{k+1}\right)V^\e(x_{k+1})\right)p_{t-s_n}(x_n,y)\dd \mathbf{x}_n\dd \mathbf{t}_n\\
&=p_t(x,y)+\\
&+{h^\e(x)}\sum_{n\geq 1}\lambda(\e)^n\int_{\Lambda_n([0,t])}\int_{\R^{3n}}p^{h,\e}_{s_1}(x,x_1)\prod_{k=1}^{n-1}\left(p^{h,\e}_{s_{k+1}-s_k}(x_k,x_{k+1})\right)\frac{{p_{t-s_n}(x_n,y)}}{h^\e(x_n)}\dd \mathbf{x}_n\dd \mathbf{t}_n,
\end{align*}
where we define \begin{align*}
p^{h,\e}_t(x,y)=-\frac{1}{h^\e(x)}p_t(x,y)V^\e(y)h^\e(y)\quad  \text{for $x,y\in\R^3$ and $t>0$},
\end{align*} 
$\mathbf{t}_n=(t_1,\dots,t_n)\in [0,\infty)^n$, $\mathbf{x}_n=(x_1,\dots,x_n)\in \left(\R^{3}\right)^n$, and
\begin{align*}
\Lambda_n([0,t])=\{\mathbf{t}_n=(t_1,\dots,t_n)\in [0,t]^n;0\leq t_1<\dots<t_n<t\}.
\end{align*}
Then, $p^{h,\e}_\cdot(x,\cdot)$ is a probability density function on $(0,\infty)\times \R^3$ for each $x\in \R^3$ by definition.

Now, let us introduce a $[0,\infty)\times \R^3$-valued Markov chain $\eta^{\e,x}=\{\eta^{\e,x}_n\}_{n\geq 0}$. It is characterized as follows: \begin{itemize}
\item (Initial condition) $\IP(\eta_0^{\e,x}=(0,x))=1$.
\item (Transition kernel) For $(s,x)\in [0,\infty)\times \R^3$, \begin{align*}
\IP\left(\eta_{n+1}^{\e,x}\in (s+\dd t,\dd y)|\eta^{\e,x}_n=(s,x)\right)=p_t^{h,\e}({x,y})\dd t\dd y .
\end{align*}
\end{itemize}
We denote the time and the spatial coordinate of $\{\eta^{\e,x}_n\}$ by $T_n^{\e,x}$ and $S^{\e,x}_n$, respectively. The Markov chain $\eta$ can be regarded as a renewal-like process. For later convenience, we allow  renewal times to take $\infty$,  that is $t_n=\infty$ for some $n\geq 1$.

For given renewal points $\{(t_n,x_n)\}_{n= 0,\dots,N}$ with $0=t_0<t_1<\dots<t_{N}=\infty$ and $x_0=x$, $x_n\in\R^3$ ($n\leq N\leq \infty$), we define a process  as follows: Let $\{B(t):t\geq 0\}$ be a standard Brownian motion in $\R^3$. Let $\{B^{\mathrm{Bri}}_n(t):0\leq t\leq 1\}_{n\geq 1}$ be an i.i.d.~sequence of standard Brownian bridges in $\R^3$. Here each $B^{\mathrm{Bri}}_n$ is a bridge from $0$ at time $0$ to $0$ at time $1$. We assume that $B$ and $\{B^{\mathrm{Bri}}_n\}$ are independent. Then we construct a process $\omega$ by \begin{align}
\omega(s)=x_n\left(1-\frac{s-t_n}{t_{n+1}-t_n}\right)+x_{n+1}\frac{s-t_n}{t_{n+1}-t_n}+\sqrt{t_{n+1}-t_n}B_n^{\mathrm{Bri}}\left(\frac{s-t_n}{t_{n+1}-t_n}\right)\label{eq:omega1}
\end{align}
for $t_n\leq s\leq t_{n+1}$ for some $n\geq 0$ and \begin{align}
\omega(s)=x_{N-1}+B(s-t_{N-1})\label{eq:omega2}
\end{align}
for $t_{N-1}\leq s<\infty$.
 This means that $\omega(s)$ is obtained by concatenating Brownian bridges from $x_n$ to $x_{n+1}$ on $[t_n,t_{n+1}]$ up to time $t_{N-1}$ and standard Brownian motion after the last renewal time. 

We denote by ${W}^{\e,x}$ the path $\omega$ determined by the renewal points $\eta^{\e,x}=\{\eta^{\e,x}_n\}_{n\geq 0}$, and by $\mathbb{P}^{\e,x}$ the joint law of $\eta^{\e,x}$ and ${W}^{\e,x}$.  

We remark that the construction of $\omega$ depends only on the renewal points.


\vskip\baselineskip

Now, we introduce a family of probability measures $\mathbb{P}_{\e,t,x}$ and a family of  random variables ${U}^{\e,t,x}$ for $\e>0$, $t>0$, $x\in\R^3$. In particular, $U^{\e,t,x}$ will determine the numbers of ``renewal times" of $\{W^{\e,t,x}\}$.  

Let $U^{\e,t,x}$ be an $\N_0$-valued random variable with \begin{align*}
\IP_{\e,t,x}(U^{\e,t,x}=n)=\begin{cases}
\dis \frac{1}{G_{\e,t}(x)},\quad &\text{for }n=0\\
\dis \frac{1}{G_{\e,t}(x)}\lambda(\e)^nh^\e(x)\IE\left[\frac{1\{T^{\e,x}_n\leq t\}}{h^{\e}(S_n^{\e,x})}\right],\quad &\text{for }n\geq1.
\end{cases}
\end{align*}

Then, we introduce a new ``renewal points" $\eta^{\e,t,x}=\left\{\eta_k^{\e,t,x}\right\}_{k= 0,\dots,{U}^{\e,t,x}}$ with $\eta^{\e,t,x}_0=(0,x)$ whose conditional law on $\{U^{\e,t,x}=n\}$ has the density given by \begin{align}
&\IP_{\e,t,x}\left(\left.\eta^{\e,t,x}_{i}\in (\dd t_i,\dd x_i), (i=1,\dots,n)  \right|U^{\e,t,x}=n\right)\notag\\
&=\frac{1}{\IP_{\e,t,x}(U^{\e,t,x}=n)}\frac{\lambda(\e)^nh^\e(x)p^{h,\e}_{t_1}(x,x_1)}{G_{\e,t}(x)h^\e(x_n)}\prod_{k=1}^{n-1}\left(p^{h,\e}_{t_{k+1}-t_k}(x_k,x_{k+1})\right)\dd \mathbf{x}_n\dd \mathbf{t}_n,\label{eq:etadensity}
\end{align}
for $\mathbf{t}_n\in\Lambda_n([0,t])$ and $\mathbf{x}_n\in\R^{3n}$. We denote the time and the spatial coordinate of $\{\eta^{\e,t,x}_n\}$ by $T_n^{\e,t,x}$ and $S^{\e,t,x}_n$, respectively. 

Given $\{(T^{\e,t,x}_{n},S^{\e,t,x}_n)\}_{n=0}^{U^{\e,t,x}}$, we construct the continuous path ${W}^{\e,t,x}$ by \eqref{eq:omega1} and \eqref{eq:omega2}. 
That is, if ${U}^{\e,t,x}=0$ so there is no renewal point ($t_{1}=\infty$), then $W^{\e,t,x}(s)=x+B_s$ for $s\geq 0$ and, if ${U}^{\e,t,x}=n \geq 1$ and 
 $\{\eta^{\e,t,x}_m:m= 0,\dots ,U^{\e,t,x}=n,n+1\}=\{(t_m,x_m):m=0,\dots,U^{\e,t,x}=n,n+1\}$ ($t_{n+1}=\infty$) are given, then $W^{\e,t,x}$ is defined by \eqref{eq:omega1} for $s\leq t_n$ and  by \eqref{eq:omega2} for $s\geq t_n$.

Thus, under $\mathbb{P}_{\e,t,x}$, we introduced quadruple of random variables $(\mathcal{U}^{\e,t,x}, \eta^{\e,t,x}, B,\{B^{\mathrm{Bri}}_n\})$ which determine the process $W^{\e,t,x}$.

\begin{prop}\label{prop:muPrepre}
For any $\e>0$, $t\geq 0$, $x\in \R^3$, $C([0,\infty),\R^3)$-valued process ${W}^{\e,t,x}$ has the law $\mu_{\e,t,x}$.
\end{prop}

\begin{proof}
Since $W^{\e,t,x}$ takes values in $C([0,\infty),\R^3)$, it is enough to see the finite-dimensional distributions of $W^{\e,t,x}$. 
For the simplicity of our argument, we consider only the one-dimensional distribution.

Let $0<s\leq t$ and $A\in \mathcal{B}(\R^3)$. Then, we have \begin{align*}
\mathbb{P}_{\e,t,x}\left(W^{\e,t,x}(s)\in A		\right)=\sum_{n\geq 0}\mathbb{P}_{\e,t,x}\left(W^{\e,t,x}(s)\in A	,U^{\e,t,x}=n	\right).
\end{align*}

We can see that \begin{align*}
\mathbb{P}_{\e,t,x}\left(W^{\e,t,x}(s)\in A	,U^{\e,t,x}=0	\right)&=\mathbb{P}_{\e,t,x}(U^{\e,t,x}=0)\mathbb{P}_{\e,t,x}\left(\left.W^{\e,t,x}(s)\in A	\right|U^{\e,t,x}=0	\right)\\
&=\frac{1}{G_{\e,t}(x)}\mathbb{P}^x(B_{s}\in A)=\frac{1}{G_{\e,t}(x)}\int_{A}p_s(x,y)\dd y.
\end{align*}
Suppose  $U^{\e,t,x}=n$ for $n\geq 1$. 
Then, 
we have \begin{align*}
&\mathbb{P}_{\e,t,x}\left(W^{\e,t,x}(s)\in A	,U^{\e,t,x}=n	\right)\\
&=\mathbb{P}_{\e,t,x}(U^{\e,t,x}=n)\mathbb{P}_{\e,t,x}\left(\left.W^{\e,t,x}(s)\in A	\right|U^{\e,t,x}=n	\right)\\
&=\int_{\Lambda_n([0,t])}\int_{\R^{3n}}\dd \mathbf{t}_n\dd \mathbf{x}_n \mathbb{P}_{\e,t,x}(U^{\e,t,x}=n)
\mathbb{P}_{\e,t,x}\left(\left.\{\eta_{i}^{\e,t,x}\}_{i=1}^n=\{(t_i,x_i)\}_{i=1}^n,W^{\e,t,x}(s)\in A	\right|U^{\e,t,x}=n	\right).
\end{align*}
We split $\Lambda_n([0,t]) $ into $n+1$ parts: \begin{align*}
\Lambda_n^k([0,t])&:=\{(t_1,\dots,t_n)\in \Lambda_n([0,t]): t_{k-1}<s\leq t_{k}\} ,\quad &k=1,\dots, n\\
\Lambda_n^{n+1}([0,t])&:=\{(t_1,\dots,t_n)\in \Lambda_n([0,t]): t_n<s\}.
\end{align*} 
Then,  we can see that for $((t_1,x_1),\dots,(t_n,x_n))\in \Lambda^k_n([0,t])\times \R^{3n}$ ($1\leq k\leq n$)\begin{align*}
&\mathbb{P}_{\e,t,x}\left(\left.W_{\e,t,x}(s)\in A	\right|U^{\e,t,x}=n,\{\eta_{i}^{\e,t,x}\}_{i=1}^n=\{(t_i,x_i)\}_{i=1}^n	\right)\\
&=\int_{A}\frac{p_{s-t_{k-1}}(x_{k-1},y)p_{t_k-s}(y,x_k)}{p_{t_{k}-t_{k-1}}(x_{k-1},x_k)}\dd y\\
&=\int_{A}\frac{p_{s-t_{k-1}}(x_{k-1},y)p_{t_k-s}(y,x_k)(-V^\e(x_k))h^\e(x_k)}{p^{h,\e}_{t_{k}-t_{k-1}}(x_{k-1},x_k)h^\e(x_{k-1})}\dd y,
\end{align*}
and that   
\begin{align*}
&\mathbb{P}_{\e,t,x}\left(\left.W^{\e,t,x}(s)\in A	\right|U^{\e,t,x}=n,\{\eta_{i}^{\e,t,x}\}_{i=1}^n=\{(t_i,x_i)\}_{i=1}^n	\right)\\
&=\int_{A}p_{s-t_n}(x_n,y)\dd y
\end{align*}
for $((t_1,x_1),\dots,(t_n,x_n))\in \Lambda^{n+1}_n([0,t])\times \R^{3n}$.
Thus, combining this with \eqref{eq:etadensity}, we have \begin{align*}
&\mathbb{P}_{\e,t,x}\left(W^{\e,t,x}(s)\in A	,U^{\e,t,x}=n	\right)\\
&=\sum_{k=1}^{n}\int_{\Lambda^{k}_n([0,t])}\dd \mathbf{t}_n\int_{\R^{3n}}\dd \mathbf{x}_n\int_A \dd y  \frac{\lambda(\e)^nh^\e(x)p^{h,\e}_{t_1}(x,x_1)}{G_{\e,t}(x)h^\e(x_n)}\prod_{k=1}^{n-1}\left(p^{h,\e}_{t_{k+1}-t_k}(x_k,x_{k+1})\right)\\
&\hspace{8em}\times \frac{p_{s-t_{k-1}}(x_{k-1},y)p_{t_k-s}(y,x_k)(-V^\e(x_k))h^\e(x_k)}{p^{h,\e}_{t_{k}-t_{k-1}}(x_{k-1},x_k)h^\e(x_{n-1})}\\
&+\int_{\Lambda^{n+1}_n([0,t])}\dd \mathbf{t}_n\int_{\R^{3n}}\dd \mathbf{x}_n\int_A \dd y  \frac{\lambda(\e)^nh^\e(x)p^{h,\e}_{t_1}(x,x_1)}{G^{\e,t}(x)h^\e(x_n)}\prod_{k=1}^{n-1}\left(p^{h,\e}_{t_{k+1}-t_k}(x_k,x_{k+1})\right)\\
&\hspace{8em}\times p_{s-t_n}(x_n,y).
\end{align*} 
Then, we  sum it over $n$ and interchange the order of the summations and integrations as\begin{align*}
\sum_{n\geq 0}\sum_{k=1}^{n+1} \int_{\Lambda^{k}_n([0,t])}\to \sum_{k\geq 1}\sum_{n=k-1}^\infty\int_{\Lambda_n^{k}([0,t])}. 
\end{align*}
Splitting $\Lambda_n^{k}([0,t])$ into two parts $\Lambda_k([0,s]) $ and $\Lambda_{n+1-k}([s,t])$ and summing over $n$, we can obtain that \begin{align*}
&\mathbb{P}_{\e,t,x}\left(W^{\e,t,x}(s)\in A	,U^{\e,t,x}=n	\right)\\
&=\frac{1}{G_{\e,t}(x)}\int_A \dd y\DE^x\left[\exp\left(\int_0^s V^\e(\omega(u))\dd u\right);\omega_s\in \dd y\right]\DE^y\left[\exp\left(\int_0^ {t-s}V^\e(\omega(u))\dd u\right)\right]\\
&=\frac{1}{G_{\e,t}(x)}\DE^x\left[\exp\left(\int_0^t V^\e(\omega(u))\dd u\right); \omega(s)\in A\right]=\mu_{\e,t,x}(\omega(s)\in A)
\end{align*}
\end{proof}

By Proposition \ref{prop:muPrepre} and \cite[Theorem 2.3]{CKMV10},  the weak convergence limit of $W^{\e,t,x}$ exists and its law is given by $\overline{P}_{\gamma,t}^x$. Thus, it is enough to investigate the weak limit of $W^{\e,t,x}$.

\section{Limit of $W^{\e,x}$}
In this section, we study the limit process of not $W^{\e,t,x}$ but $W^{\e,x}$.

We focus on the set of renewal points $\eta^{\e,x}$. If $\eta^{\e,x}$ converges to some point sets in some sense, then we can expect the $W^{\e,x}$ converges to the process by concatenating the Brownian brides. In this section, we prove rigorously that this intuition is true.

We introduce a metric  for the family of all closed sets in $[0,\infty)$ which is compatible to so-called Fell-Matheron topology in subsection \ref{subsec:FelMath}. In subsection \ref{subsec:conveta}, we  prove the convergence of $\eta^{\e,x}$. Then, we prove the convergence of the process $W^{\e,x}$ in subsection \ref{subsec:Convomegaex}.

\subsection{Fell-Matheron topology} \label{subsec:FelMath}


In this section, we introduce the Fell-Matheron topology which is often used in the literature of the renewal process and pinning models\cite{Ber97,CGZ06,CSZ16,DO20,FFM85,Gia07}. However, we will not give the definition of Fell-Matheron topology but alternatively we introduce a metric compatible to the Fell-Matheron topology. (For the definition of the Fell-Matheron topology, the reader may refer to \cite{Mat75}.)

Also, we collect some theorems used in this paper from \cite{CSZ16,Mat75,WW07}. 

\vskip\baselineskip

Let $\mathcal{C}$ be the family  of all closed subsets of $[0,\infty)$. We introduce a metric for $A,B\in \mathcal{C}$ in the following way. We consider $[0,\infty]$ with the metric given by \begin{align}
d(x,y)=|\arctan x-\arctan y|\quad x,y\in [0,\infty].\label{eq:distance}
\end{align}

Let $2^{[0,\infty]}=\{C\subset [0,\infty];C\text{ is \textit{non-empty} closed subset of }[0,\infty]\}$.

Then, we can define the Hausdorff distance of $K,K'\in 2^{[0,\infty]}$ by \begin{align*}
d_H(K,K'):=\max\left\{\sup_{x\in K}d(x,K'),\sup_{x'\in K'}d(x',K)		\right\},
\end{align*}
 where $d(x,A)=\inf_{y\in A}d(x,y)$. Considering a map $c:\mathcal{C}\to 2^{[0,\infty]}$ by $c(C)\mapsto C\cup \{\infty\}$, we can introduce a metric $d_{\mathrm{FM}}$ on $\mathcal{C}$. It is known that this topology is compatible to the Fell-Mathron topology \cite[Theorem 1]{WW07}. 
 
\begin{theorem}\label{thm:FMPolish}\cite[Theorem 1.2.1]{Mat75} $(\mathcal{C},d_{\mathrm{FM}})$ is compact and hence is a Polish space.
\end{theorem} 
 The following is a consequence from the definition of the topology (see also \cite[p.6]{Mat75} and \cite[Remark A.1]{CSZ16}).
\begin{lemma}\label{lem:convC}
$C_n\to C$ in $\mathcal{C}$ if and only if the following conditions hold:
\begin{enumerate}
\item\label{item:interOpen} for any open set $G\subset [0,\infty)$ with $G\cap C\not=\emptyset$, $G\cap C_n\not=\emptyset$ for large $n$;
\item\label{item:interComp} for any compact set $K\subset [0,\infty)$ with $K\cap C=\emptyset$,  $K\cap C_n=\emptyset$ for large $n$.
\end{enumerate}
\end{lemma}

Also, we have the following lemma.
\begin{lemma}\label{lem:Cconveuqiv}\cite[Theorem 1-2-2]{Mat75}
$C_n\to C$ in $\mathcal{C}$ if and only if the following conditions hold:
\begin{enumerate}
\item if $x_n \in C_n$ is a sequence with $x_n\to x\in [0,\infty)$, then $x\in C$;
\item if $n_k\to C_{n_k}$ is a subsequence, every convergent sequence $n_k\to x_{n_k}\in F_{n_k}$ satisfies $\lim_{k\to\infty}x_{n_k}\in F$.
\end{enumerate}

\end{lemma}

Thus, we can equip $\mathcal{C}$ with a nice metric. However, to see the convergence of $\{\eta^{\e,x}_n\}$, we will use not the metric $d_{\mathrm{FM}}$ directly.

Our first main  result in this section is the following.

\begin{lemma}\label{lem:STconv}
$\mathcal{C}$-valued random variables $T^{\e,x}=\{T^{\e,x}_n\}_{n\geq 0}$ weakly converges to a random variable $\tau^x$ on $[0,\infty)$ and for any $T>0$, $\sup\{|S_n^{\e,x}|;n\geq 1, T_n^{\e,x}\leq T \}\to  0$ in probability.
\end{lemma}

To prove the convergence of the renewal points, we will use the associated ``subordinator".


Let $\mathcal{D}$ be a set of c\`adl\`ag and nondecreasing functions on $[0,\infty)$ defined by 
\begin{align*}
\mathcal{D}:=\left\{f:[0,\infty)\to \R; \begin{array}{l}
\exists\lim_{s\nearrow t}f(s), \quad \text{ for all $t\in[0,\infty)$},\\
f(t)=\lim_{s\searrow t}f(s),\quad \text{ for all $t\in [0,\infty)$},\\
f(s)\leq f(t),\quad \text{ for all $0\leq s\leq t<\infty$}\\
f(0)=0,\quad \lim_{t\to \infty}f(t)=\infty
\end{array}			\right\}
\end{align*}
and we equip $\mathcal{D}$ with Skorokhod topology determined by the metric \begin{align*}
d_{\mathrm{S}}(f,g)=\sum_{M\geq 1}\frac{1}{2^M}\left(\|f^M-g^M\|_{M}\wedge 1\right),\quad \text{for $f,g\in \mathcal{D}$}
\end{align*}
where we define \begin{align*}
f^M(x)=\begin{cases}
f(x),\quad &x\in [0,M-1]\\
(M-x)f(x),\quad &x\in [M-1,M]\\
0,\quad &x\geq M
\end{cases}
\end{align*}
for $f\in \mathcal{D}$ and 
we define \begin{align*}
\|f-g\|_M=\inf_{\lambda}\max\left\{\sup_{s\in [0,M]}|\lambda(s)-s|,\sup_{s\in [0,M]}|f(s)-g(\lambda(s))|\right\},
\end{align*}
where we take infimum over all strictly increasing, continuous, and bijections $\lambda:[0,M]\to [0,M]$, where we denote by $\Lambda_M$ such functions. 

When we focus on the closure of range of $f\in \mathcal{D}$, we have the following.
\begin{lemma}\label{lem:SkotoFM}
Let  $f_n$ ($n\geq 1$) and $f$ be functions in  $\mathcal{D}$. Suppose that $f_n$ converges to $f$ in $(\mathcal{D},d_{\mathrm{S}})$ and that \begin{align}
\text{for any $L>0$, there exists an $R>0$ such that $f(R)\geq L$}.\label{eq:assmp}
\end{align} 

Then $R_{f_n}:=\overline{f_n([0,\infty))}$ converges to  $R_f:=\overline{f_n([0,\infty))}$ in $(\mathcal{C},d_{\mathrm{FM}})$.
\end{lemma}
\begin{rem}
\eqref{eq:assmp} is necessary. Indeed, we take \begin{align*}
f_n(x)=\begin{cases}
0,\quad &x\in [0,n]\\
x-n,\quad &x\geq n.
\end{cases}
\end{align*}
Then, we can see that $f_n\to 0$ in $\mathcal{D}$ but $R_{f_n}=[0,\infty)\not=R_0$.
\end{rem}

\begin{proof}

Suppose $d_S(f_n,f)\to 0$. 

We will check conditions in Lemma \ref{lem:convC}. 

\eqref{item:interOpen}. Let $O$ be an open set in $[0,\infty)$ with $R_{f}\cap O\not=\emptyset$. Then, there exist a $y\in R_f\cap O$ and $\{x_n\}_{n\geq 1}\subset  [0,\infty)$ such that $f(x_n)\to y$. We take $\delta\in (0,1)$ such that $\{x\in [0,\infty);|x-y|<\delta\}\subset O$. Then,  there exists an $N\geq 1$ such that for any $n\geq N$, $|f(x_n)-y|<\frac{\delta}{3}$. 


Let $M>0$ be an integer such that $\sup_{n\geq 1}x_n+2\leq M$.  Then, there exists a $\widetilde{N}\geq 1$ such that for any $n\geq \widetilde{N}$, $\|f_n-f\|_M<\frac{\delta}{3}$. Also, for $n\geq \widetilde{N}$, there exists a $\lambda_n\in \Lambda_M$ such that \begin{align*}
\max\left\{\sup_{x\in [0,M]}|\lambda_n(x)-x|,\sup_{x\in [0,M]}|f^M(x)-f^M_n(\lambda_n(x))|\right\}<\frac{2\delta}{3}.
\end{align*} 
In particular, we have \begin{align*}
|\lambda_n(x_n)-x_n|<\frac{2\delta}{3}\quad \text{and }|f^M(x_n)-f^M_n(\lambda_n(x_n))|<\frac{2\delta}{3}
\end{align*}
 for $n\geq \max\{N,\widetilde{N}\}$. By definition,  $x_n\leq M-2$ and $\lambda_n(x_n)\leq M-1$ and hence, $|f_n(\lambda_n(x_n))-y|<\delta$. Thus, $R_{f_n}\cap O\not=\emptyset$.
 
 \eqref{item:interComp}. Let $K$ be a compact set in $[0,\infty)$ with $K\cap R_f=\emptyset$. We may set  $\delta=\inf\{|x-y|;x\in K,y\in R_f\}\in (0,1)$.

We consider \begin{align*}
s_0=\sup\{s\geq 0; f(s)\leq \min K\}=\inf\{s\geq 0; f(s)\geq \max K\},
\end{align*}
where the coincidence of the infimum and the supremum, and its strict positivity follow from the definition of $\mathcal{D}$. Also, it is bounded by  \eqref{eq:assmp}. In particular, \begin{align*}
\max\left\{\min K-\lim_{s\nearrow x_0}f(s), f(x_0)-\max K\right\}=\delta. 
\end{align*} 

For fixed $L\geq 0$ with $K\cap [L-1,\infty]=\emptyset$, we take $R>0$ as in \eqref{eq:assmp}.

There exists an ${N}\geq 1$ such that for any $n\geq {N}$, there exists  a $\lambda_n\in \Lambda_{s_0+2}$ such that \begin{align*}
\max\left\{\sup_{x\in [0,s_0+2]}|\lambda_n(x)-x|,\sup_{x\in [0,s_0+2]}|f^{s_0+2}(x)-f^{s_0+2}_n(\lambda_n(x))|\right\}<\frac{\delta}{2}.
\end{align*}
This implies that for any $x\in f_n([0,x_0+2])$, $\inf \{|x-y|;y\in R_f\}<\frac{\delta}{2}$ and $f_n(R+2)\geq f(s_0)-\frac{\delta}{2}$. Therefore, for any $n\geq N $, $\overline{f_n([0,s_0+2])}\cap K=\emptyset$ and $\overline{f_n([0,\infty))}\cap K=\emptyset$. 

\end{proof}

Thus, we can  obtain the following corollary by using the Skorokhod representation theorem  for Polish space $(\mathcal{D},d_{\mathrm{S}})$.
\begin{cor}\label{cor:SkotoMat}
Let $\{f_n\}_{n\geq 1}$ and $f$ be $\mathcal{D}$-valued random variables on some probability space $(\Omega,\mathcal{F},P)$.  Suppose that $f_n$ weakly converges to $f$  in $(\mathcal{D},d_\mathrm{S})$-sense and $f$ satisfies \eqref{eq:assmp} $P$-a.s. 
Then, $R_{f_n}$ weakly converges to $R_f$ in $(\mathcal{C},d_{\mathrm{FM}})$.
\end{cor}

We define  a c\`adl\`ag and nondecreasing function $\left\{\widetilde{T}_s^{\e,x}\right\}_{s\geq 0}$ by \begin{align*}
\widetilde{T}_s^{\e,x}=\begin{cases}
0,\quad &s\in \left[0,1\right)\\
T_1^{\e,x},\quad &s\in [1,a_\e+1)\\
T_{n+1}^{\e,x},\quad &s\in [na_\e+1,(n+1)a_\e+1), \text{ for some $n\geq 1$}
\end{cases}
\end{align*}
where we set \begin{align*}
a_\e=\int_0^1 \int_{\R^3}sp_s^{h,\e}(0,x)\dd s\dd x\sim 2\e.
\end{align*}

Then, it is trivial that $\{T_n^{\e,x}\}_{n\geq 0}=\overline{\left\{\widetilde{T}_{s}^{\e,x};s\in [0,\infty)\right\}}$. We will prove the convergence  of $\{T_n^{\e,x}\}_{n\geq 0}$ via the convergence of the process $\left\{\widetilde{T}_s^{\e,x}\right\}_{s\geq 0}$ with Corollary \ref{cor:SkotoMat}.

\subsection{Proof of Lemma \ref{lem:STconv}}\label{subsec:conveta}

\subsubsection{Markov chain associated with $S_n^{\e,x}$}

We introduce  a new Markov chain $\{M^\e_n\}_{n\geq 0}$ on $\R^3$ with the transition kernel \begin{align*}
\frac{p_{\e^2}(y)}{h^\e(x)}H(0,x-y)
\end{align*}
 and  we denote by $\mathtt{P}_x$ and $\mathtt{E}_x[\cdot]$  the law of $M_n^\e$ with $M_0^\e=x\in \R^3$ and the expectation w.r.t.~$\mathtt{P}_x$.

We will see at Lemma \ref{lem:MScorr} that $M^{\e}$ helps us to analyze of $\{S_n^{\e,x}\}_{n\geq 0}$. 

We note that \begin{align*}
\int_{a}^\infty p_t^{h,\e}(x,y)\dd t=\int_a^\infty \frac{1}{h^\e(x)}p_t(x,y)V^\e(y)h^{\e}(y)\dd t =\frac{p_{\e^2}(y)}{h^\e(x)}H(a,x-y).
\end{align*}
 
By definition,  $\{M^\e_n\}_{n\geq 1}$  is a  reversible Markov chain with invariant measure and invariant distribution  \begin{align}
&\rho^\e(\dd x)=\rho^\e(x)\dd x=h^\e(x)^2V^\e(x)\dd x=p_{\e^2}(x)h^\e(x)\dd x=p_{\e^2}(x)\dd x\int_0^\infty\dd t \int_{\R^3}p_t(x,y)p_{\e^2}(y)\dd y\label{eq:invariantmeasure}\\
&\pi^\e(\dd x)=\pi^\e(x)\dd x:=2\pi^{\frac{3}{2}}\e \rho^\e(x)\dd x,\label{eq:invariantdistribution}
\end{align}
respectively.

Moreover, $\{M_n^\e\}_{n\geq 0}$ has the scaling property in the following sense:
\begin{lemma}\label{lem:scalingMCM}
Suppose the Markov chain $\{M_n^1\}_{n\geq 0}$ starts at $M_0^1=x\in \R^3$. Then $\{\e M_n^1\}_{n\geq 0}$ is identically distributed as $\{M_n^\e\}_{n\geq 0}$ with $M_0^\e=\e x$.

\end{lemma}

\begin{proof}
It is enough to show that the transition density kernel of $\e M_n^1$ coincides with the one of $M_n^\e$.  For $A\in \mathcal{B}(\R^3)$ and $z\in \R^3$, \begin{align}
\mathtt{P}_x(\e M_n^1\in A|\e M_{n-1}^1= z)&=\int_{\frac{A}{\e}} \frac{p_1(y)}{h\left(\frac{z}{\e}\right)}H\left(0,\frac{z}{\e}-y\right)\dd y\notag\\
&=\e^{-3}\int_{{A}} \frac{p_1\left(\frac{\tilde{y}}{\e}\right)}{h\left(\frac{z}{\e}\right)}H\left(0,\frac{z}{\e}-\frac{\tilde{y}}{\e}\right)\dd {\tilde{y}}.\label{eq:scalingtrker}
\end{align}

Since we know \begin{align*}
&h\left(\frac{z}{\e}\right)=\e h^\e(z), \quad \frac{1}{\e^3}p_1\left(\frac{\tilde{y}}{\e}\right)=p_{\e^2}(\tilde{y}),\quad \text{and }H\left(0,\frac{z-\tilde{y}}{\e}\right)=\frac{\e}{2\pi|z-\tilde{y}| }=\e H(0,z-\tilde{y}),
\end{align*}
\eqref{eq:scalingtrker} is equal to $\int_A \frac{p_{\e^2}(\tilde{y})}{h^\e(z)}H(0,z-\tilde{y})\dd \tilde{y}=\mathtt{P}_{\e x}(M_n^\e\in A|M_{m-1}^\e=z)$. 
\end{proof}

In the same way, we find that $\pi^\e(x)=\e^{-3} \pi^1\left(\frac{x}{\e}\right)$.


\begin{lemma}\label{lem:MScorr}
Suppose  $M_0^{\e}=x\in \R^3$. Then, we have \begin{align*}
\mathtt{E}_x\left[f\left(M_n^\e\right)\right]=\mathbb{E}\left[f\left(S_n^{\e,x}\right)\right]
\end{align*}
for any $f\in B_b(\R^3)$ and $n\geq 1$. 
\end{lemma}
\begin{proof}
By definition, \begin{align*}
\mathbb{E}\left[f\left(S_n^{\e,x}\right)\right]&=\int \dd \mathbf{t}\int \dd \mathbf{x} f(x_n)\prod_{i=1}^n p_{t_i}^{h,\e}(x_{i-1},x_i)\\
&=\int \dd \mathbf{x} f(x_n)\prod_{i=1}^n \frac{p_{t_i}^{h,\e}(x_i)}{h^\e(x_{i-1})}H(0,x_i-x_{i-1})\\
&=\int \dd \mathbf{x} f(x_n)\prod_{i=1}^n \mathtt{P}\left(\left.M^\e_{i}\in \dd x_i\right|M_{i-1}^\e=x_{i-1}\right)\\
&=\mathtt{E}_x\left[f(M_n^\e)\right].
\end{align*}
\end{proof}

\begin{rem}\label{rem:Verg}
 $\{M_n^1\}_{n\geq 0}$ is $V$-uniform ergodic with $V(x)=(|x|\vee 1)^p$ for $p\geq 0$ so that we can control the rate of convergence of expectation of $M_n^1$ to the expectation with respect to $\pi^1$. (See Appendix \ref{app:Vuni}.)
\end{rem}

From Lemma \ref{lem:scalingMCM} and Lemma \ref{lem:MScorr}, we may expect $S_{n}^{\e,x}$ cannot be located away from the origin. This is true (Corollary \ref{cor:Sstay} below).  

We set \begin{align*}
A_\e(k,R)=\bigcap_{l=1,\dots,k}\left\{|M^\e_l|< R\right\}
\end{align*}
for $k\geq 1$, $R>0$.

\begin{prop}\label{prop:Mfar'}
There exists $C>0$ such that  for $x\in\R^3$, $k\geq 1$, $R>0$,\begin{align*}
\mathtt{P}_{x}\left(\left(A_\e(k,R)\right)^c\right)\leq C\frac{\left({\max\{|x|, \e\}}+k\max\{R,\e\}\right)}{\e}\frac{\max\{R,\e\}}{\e}{e^{-\frac{R^2}{2\e^2}}} .
\end{align*}
\end{prop}

\begin{cor}\label{cor:Sstay}
For each $x\in \R^3$ and $t>0$, \begin{align*}
\mathbb{P}\left(\sup_{1\leq n\leq t\e^{-2}}|S_{n}^{\e,x}|\geq \sqrt{\e}\right)\to 0.
\end{align*}
\end{cor}

\begin{proof}[Proof of Corollary \ref{cor:Sstay}]
From Lemma \ref{lem:MScorr}, we have \begin{align*}
\mathbb{P}\left(\left|S_{n}^{\e,x}\right|\geq \sqrt{\e}\right)=\mathtt{P}_x\left(|M_n^\e|\geq \sqrt{\e}\right)\leq \mathtt{P}_x\left(A_\e(n,\sqrt{\e})^c\right). 
\end{align*}
We have from Proposition \ref{prop:Mfar'} \begin{align*}
\mathbb{P}\left(\sup_{1\leq n\leq t\e^{-2}}|S_{n}^{\e,x}|\geq \sqrt{\e}\right)&\leq \sum_{l=1}^{\frac{t}{\e^2}}\mathtt{P}_x\left(A_\e(n,\sqrt{\e})^c\right)\\
&\leq Ct^{-\frac{7}{2}}(\max\{|x|,\e\}+t\e^{-\frac{3}{2}})e^{-\frac{1}{2\e}}.
\end{align*}
\end{proof}

From  Corollary \ref{cor:Sstay}, the proof of convergence of $\{S_n^{\e,x}\}$ in Lemma \ref{lem:STconv} is almost completed. To complete the proof, we need to give an estimate of $\sup\{n\geq 0: T_{n}^{\e,x}\leq t\}$.

\begin{proof}[Proof of Proposition \ref{prop:Mfar'}]
We introduce a stopping time $\tau_M^{\e,R}:=\inf\left\{l\geq 1;|M_l^\e|\geq R\right\}$. Then, \begin{align*}
\mathtt{P}_{x}\left(\left(A_\e(k,R)\right)^c\right)=\mathtt{P}_{x}\left(\tau_M^{\e,R}\leq k\right)=\sum_{l=1}^{k}\mathtt{P}_{x}\left(\tau_M^{\e,R}=l\right).
\end{align*}
We define \begin{alignat*}{2}
&\overline{q}(\e,R):=\sup_{|z|\leq R}\mathtt{P}_z\left(M^\e_1\geq R\right),\qquad 	
&&{q}(x,\e,R):=\mathtt{P}_x\left(M^\e_1\geq R\right),
\end{alignat*}
and then we can see from the Markov property that 
\begin{align*}
&\mathtt{P}_{x}\left(\tau_M^{\e,R}=l\right)\leq \begin{cases}
q(x,\e,R),\quad &\l=1\\
\overline{q}(\e,R),\quad &l\geq 2.
\end{cases}
\end{align*}

By using \eqref{eq:Hprop5}, we have
\begin{align*}
&\overline{q}(\e,R)\\
&\leq \sup_{|z|\leq R}\int_{|y|\geq R}\frac{p_{\e^2}(y)}{h^\e(z)}H(0,y-z)\dd y\\
&\leq \sup_{|z|\leq R}\frac{1}{h^\e(z)}\left(\int_{|y|\geq R,|y-z|\geq \frac{\e}{2}}{p_{\e^2}(y)}H(0,y-z)\dd y+\int_{|y|\geq R,|y-z|\leq \frac{\e}{2}}{p_{\e^2}(y)}H(0,y-z)\dd y\right)\\
&\leq C\max\{R,\e\} \left(\int_{|y|\geq R }\frac{1}{\e}{p_{\e^2}(y)}\dd y+\int_{|y-z|\leq \frac{\e}{2}}\frac{1}{\e^3|y-z|}{e^{-\frac{R^2}{2\e^2}}}\dd y\right)\\
&\leq C\frac{\max\{R,\e\}}{\e}\left(\frac{\max\{R,\e\}}{\e}+\frac{\e}{\max\{R,\e\}}\right)e^{-\frac{R^2}{2\e^2}} ,
\end{align*}
where we have used $\int_{R}^\infty r^2e^{-\frac{r^2}{2}}\dd r\leq CRe^{-\frac{R^2}{2}}+\int_{R}^\infty \frac{r}{\max\{R,1\}}e^{-\frac{r^2}{2}}\dd r$.

The same argument yields that \begin{align*}
q(x,\e,R)\leq C\frac{\max\{|x|,\e\}}{\e}\left(\frac{\max\{R,\e\}}{\e}+\frac{\e}{\max\{R,\e\}}\right)e^{-\frac{R^2}{2\e^2}}.
\end{align*}
\end{proof}

\subsubsection{Some etimates on $\{T_n^{\e,x}\}_{n\geq 1}$}
In this subsubsection, we prove that $T_n^{\e,x}$ exceeds $T$ in $O(\e^{-1})$-times and complete the proof of convergence of $\{S_n^{\e,x}\}$ in Lemma \ref{lem:STconv}.

\begin{prop}\label{prop:hprop}
For $t>0$ and $x\in \R^3$ \begin{align}
\frac{H(t+\e^2,x)}{H(\e^2,x)}\geq \frac{H(t+\e^2,0)}{H(\e^2,0)}=\sqrt{\frac{\e^2}{t+\e^2}}\label{eq:stodom}
\end{align}
\end{prop}

\begin{proof}

We set for $\alpha\in (0,1]$, $r>0$\begin{align*}
F_\alpha(r)=\log \frac{\int_{0}^{\alpha r }u^{-\frac{1}{2}}e^{-u}\dd u}{\int_{0}^{r }u^{-\frac{1}{2}}e^{-u}\dd u}=\log \frac{H\left(\frac{1}{2\alpha},x\right)}{H\left(\frac{1}{2},x\right)}
\end{align*}
with $|x|=r$.

Then, we have \begin{align*}
\frac{\dd}{\dd r}F_\alpha(r)=\frac{\sqrt{\frac{\alpha}{r}}e^{-\alpha r}}{\int_{0}^{\alpha r }u^{-\frac{1}{2}}e^{-u}\dd u}-\frac{ \frac{1}{\sqrt{r}}e^{-r}}{\int_{0}^{ r }u^{-\frac{1}{2}}e^{-u}\dd u}.
\end{align*}
Hence, it is enough to show that \begin{align*}
\sqrt{{\alpha}}e^{-\alpha r}\int_{0}^{ r }u^{-\frac{1}{2}}e^{-u}\dd u-e^{-r}{\int_{0}^{\alpha r }u^{-\frac{1}{2}}e^{-u}\dd u}\geq 0.
\end{align*}
This holds since the left-hand side is written as \begin{align*}
\int_0^{\alpha r}\frac{1}{\sqrt{u}}\left(e^{-\alpha r-\frac{u}{\alpha}}-e^{-r-u}\right)\dd u
\end{align*}
and $\alpha r+\frac{u}{\alpha}\leq r+u$ for $u\in [0,\alpha r]$.
\end{proof}

\begin{lemma}\label{lem:Tlowerbdd}
For any $x\in \R^3$ and $t\geq 0$, we have \begin{align*}
\mathbb{P}\left(T_n^{\e,x}\leq t\right)\leq \mathbb{P}\left(\sum_{i=1}^n \tilde{\sigma}_i^{\e,0}\leq t\right),
\end{align*}
where $\left\{\tilde{\sigma}_n^{\e,0}\right\}_{n\geq 1} $ is an i.i.d.~random variable with the distribution $\mathbb{P}\left(\tilde{\sigma}_1^{\e,0}\geq 0\right)=\frac{H(t+\e^2,0)}{H(\e^2,0)}$.
\end{lemma}
\begin{proof}
\eqref{eq:stodom} yields that  for any $x\in \R^3$ and $n\geq 0$, $\sigma_n^{\e,x}$ stochastically dominates $\sigma_1^{\e,0}$. Combining this with the Markov property of $T_n^{\e,x}$, the statement holds.
\end{proof}

Now, we have the following corollary.
\begin{cor}\label{cor:Tdeciation}
For any $x\in \R^3$, $t>0$, and for any large $\lambda\geq 1$ and small $\e>0$\begin{align*}
\mathbb{P}\left(T_n^{\e,x}\leq t\right)\leq e^{\lambda t}\left(1-\frac{\lambda^\frac{1}{4}}{2}\e\right)^n.
\end{align*}
\end{cor}

\begin{proof}
By Chebyshev's inequality, we have \begin{align*}
\mathbb{P}\left(\sum_{i=1}^n \tilde{\sigma}_i^{\e,0}\leq t\right)\leq e^{\lambda t}\left(\int_0^\infty \frac{\e}{2\left(u+\e^2\right)^\frac{3}{2}}e^{-\lambda u}\dd u\right)^n
\end{align*}
for $\lambda>0$. We have \begin{align*}
\int_0^\infty \frac{\e}{2\left(u+\e^2\right)^\frac{3}{2}}e^{- \lambda u}\dd u&=\int_0^\infty \frac{1}{2\left(u+1\right)^\frac{3}{2}}e^{- \lambda  \e^2u}\dd u\\
&\leq \int_0^{\frac{1}{\e^2\sqrt{\lambda} }} \frac{1}{2\left(u+1\right)^\frac{3}{2}}\dd u+\int_{\frac{1}{\e^2\sqrt{\lambda} }}^\infty \frac{1}{2 \left(\frac{1}{\e^2\lambda }+1\right)^\frac{3}{2}}e^{-  \lambda \e^2u}\dd u\\
&=1-\frac{\e \lambda^\frac{1}{4} }{\sqrt{1+\e^2 \sqrt{\lambda} }}+\frac{\e \lambda ^{-\frac{3}{4}}e^{-\sqrt{\lambda}}}{2 \left(1+\e^2\sqrt{\lambda} \right)^\frac{3}{2}}\\
&\leq 1-\frac{\lambda^\frac{1}{4}}{2}\e
\end{align*}
for any large $\lambda\geq 1$ and  small $\e>0$.
\end{proof}

\begin{proof}[Proof of convergence of $\{S_n^{\e,x}\}$ in Lemma \ref{lem:STconv}]
We know from Corollary \ref{cor:Sstay} that $\sup_{1\leq n\leq T\e^{-2}}|S_n^{\e,x}|\to 0$ in probability as $\e\to 0$.
Also,  taking $n=T\e^{-2}$ in Corollary \ref{cor:Tdeciation}, we can see that $\sup\{n\geq 1; T_n^{\e,x}\leq T\}\leq T\e^{-2}$ in high probability.   Thus, the statement holds.	
\end{proof}

\subsubsection{Convergence of $\widetilde{T}^{\e,x}$}

As mentioned in the end of Subsection \ref{subsec:FelMath}, we will see the convergence of $\widetilde{T}^{\e,x}$.

First, we will look at the limit of $T_1^{\e,x}$.

\begin{lemma}\label{lem:limitT1}
For any $x\in \R^3$ ($x\not=0$), we have $T_1^{\e,x}\Rightarrow \sigma^x$ as $\e\to 0$, where $\sigma^x$ is a $(0,\infty)$-valued random variable whose law has the density \begin{align*}
2\pi|x| p_t(x).
\end{align*} 
Moreover, for any $t>0$ and $x\in \R^3$ ($x\not=0$), there exists  $C$ such that for $R>1$, \begin{align}
\IP_x\left(T_1^{\e,x}>t,|S_1^{\e,x}|>R\e\right)\leq C\frac{\max\{|x|,\e\}}{\sqrt{t}}\left(R+\frac{1}{R}\right)e^{-\frac{R^2}{2}} .\label{eq:S1small}
\end{align}
\end{lemma}

\begin{proof}
For $t>0$ and $R>1$, we have \begin{align*}
\IP_x\left(T_1^{\e,x}>t,|S_1^{\e,x}|>R\e\right)&=\int_t^\infty \dd s\int_{|y|>R\e}\frac{1}{h^\e(x)}p_s(x,y)p_{\e^2}(y)\dd y\\
&\leq  \int_{|y|>R\e}\frac{1}{h^\e(x)}H(t,x-y)p_{\e^2}(y)\dd y.
\end{align*}
We can find from \eqref{eq:Hprop4} and \eqref{eq:Hprop5} that there exists $C>0$ such that for each $x\not=0$ and $t>0$,  \begin{align*}
&\int_{|y|>R\e}\frac{1}{h^\e(x)}H(t,x-y)p_{\e^2}(y)\dd y\leq C\frac{\max\{|x|,\e\}}{\sqrt{t}}\int_{|y|\geq R\e}p_{\e^2}(y)\dd y\leq C\frac{\max\{|x|,\e\}}{\sqrt{t}}\left(R+\frac{1}{R}\right)e^{-\frac{R^2}{2}}. 
\end{align*} 
Thus, \eqref{eq:S1small} holds.
Also, we have
\begin{align*}
\IP_x\left(T_1^{\e,x}>t\right)&=\int_t^\infty \dd s\int_{\R^3}\frac{1}{h^\e(x)}p_s(x,y)p_{\e^2}(y)\dd y\\
&=  \int_{t}^\infty \frac{1}{h^\e(x)}p_{s+\e^2}(x)\dd s\to 2\pi |x|H(t,x)\qquad \text{as $\e\to 0$.} 
\end{align*}

Thus, we have \begin{align*}
T_1^{\e,x}\Rightarrow \sigma^x \quad \text{as }\e\to 0.
\end{align*}

\end{proof}

Thus, it is enough to the convergence of $\widetilde{T}_{s}^{\e,x}-\widetilde{T}_1^{\e,x}$.

To see this, we use the method by Durrett and Resnick \cite{DR78} where they gave a criterion for the convergence of partial sums of dependent random variable to a L\'evy process. Here, we write the statement in \cite[Theorem 2.1]{Gay12}.

\begin{theorem}\label{thm:Levy} Let  $(\Omega,\mathcal{F},P)$ be a probability space and let $\{\mathcal{F}_{n,i};n\geq 1,i\geq 0\}$ be an array of sub-sigma fields $\mathcal{F}$ such that for each $n\geq 1$, and $i\geq 1$, $Z_{n,i}$ is $\mathcal{F}_{n,i}$-measurable and $\mathcal{F}_{n,i-1}\subset \mathcal{F}_{n,i}$. Let $\{Z_{n,i};n\geq 1,i\geq 1\}$ be an array of random variables on $(\Omega,\mathcal{F},P) $ with $Z_{n,i}\geq 0$ $P$-a.s. For a non-random sequence  $a_n>0$ with $\lim_{n\to \infty}a_n=\infty$ and $t\geq 0$, we define \begin{align*}
S_{n}(t)=\sum_{i=1}^{\lfloor a_n t\rfloor}Z_{n,i}.
\end{align*}
Let $\nu$ be a $\sigma$-finite measure on $(0,\infty)$ satisfying  $\int (1\wedge x)\nu(\dd x)<\infty$, and let $\{S(t)\}_{t\geq 0}$ be the subordinator of Laplace exponent $\Phi(\theta)=\int_0^\infty (1-e^{-\theta x})\nu(\dd x)$ $(\theta\geq 0)$. 

Suppose the following three conditions hold.
\begin{enumerate}[label=(\roman*)]
\item For any $t>0$, $x>0$ with $\nu(\{x\})=0$, \begin{align*}
\sum_{i=1}^{\lfloor a_n t\rfloor}P\left(\left.	Z_{n,i}>x	\right|\mathcal{F}_{n,i-1}\right)\to t\nu((x,\infty)),\quad \textrm{in probablity}.
\end{align*}
\item For any $t>0$, and $\delta>0$, \begin{align*}
\sum_{i=1}^{\lfloor a_n t\rfloor}P\left(\left.	Z_{n,i}>\delta	\right|\mathcal{F}_{n,i-1}\right)^2\to 0,\quad \textrm{in probablity}.
\end{align*}
\item For any $t>0$ and $\delta>0$, \begin{align*}
\lim_{a\to 0}\varlimsup_{n\to\infty}E\left(	\sum_{i=1}^{\lfloor a_n t\rfloor}Z_{n,i}1\{Z_{n,i}\leq a\}		\right)\to 0.
\end{align*} 

\end{enumerate}
 Then, $S_n$ weakly converges to $S$ under the Skorokhod topology.

\end{theorem} 

We denote by $\{\sigma_n^{\e,x}\}_{n\geq 1}$ the increment of $\{T_{n}^{\e,x}\}_{n\geq 0}$, i.e.\begin{align*}
&\sigma_{n}^{\e,x}=\sigma_n:=T^{\e,x}_{n}-T^{\e,x}_{n-1},\quad \text{for }n\geq 1.
\end{align*}
We remark that $\mathbb{P}\left(\sigma_n^{\e,x}<\infty\right)=1$ for all $\e>0$, $x\in \R^3$, and $n\geq 0$. Also, we define $\mathcal{F}_n=\sigma\left[\eta_{m}^{\e,x};0\leq m\leq n\right]$ for $n\geq 0$. 

We can apply Theorem \ref{thm:Levy} to $\{\sigma_n^{\e,x}\}$ by Lemma \ref{lem:meanconv}-Lemma \ref{lem:smallissmall} below so that we obtain the following:
\begin{lemma}\label{lem:Tzeta}
\begin{align*}
\widetilde{T}_{s}^{\e,x}-\widetilde{T}_{1}^{\e,x}\Rightarrow \zeta_s,
\end{align*}
where $\{\zeta_t\}_{t\geq 0}$ is a subordinator with the L\'evy measure $\nu(\dd x)=(2x)^{-\frac{3}{2}}\dd x$. 
\end{lemma}

From Lemma \ref{lem:limitT1} and Lemma \ref{lem:Tzeta}, we find that $\widetilde{T}^{\e,x}$ weakly converges to a random process \begin{align*}
\widetilde{T}^x(s)=\begin{cases}
0,\quad &0\leq s< 1\\
\sigma^x+\zeta_{s-1},\quad &s\geq 1
\end{cases}
\end{align*} 
in $(\mathcal{D},d_\mathrm{S})$, where $\zeta_s$ is the subordinator  given in Lemma \ref{lem:Tzeta}.  It is easy to see that $\widetilde{T}^x$ satisfies \eqref{eq:assmp} and hence one complete the proof of convergence $T^{\e,x}$ in Lemma \ref{lem:STconv}: 
\begin{lemma}\label{lem:setcounvergence}
\begin{align*}
\left\{T_{n}^{\e,x}\right\}_{n\geq 1}\Rightarrow \overline{R}_{\widetilde{T}^x}
\end{align*}
in $(\mathcal{C},d_{\mathrm{FM}})$ as $\e\to 0$.
\end{lemma}

In the rest of this subsection, we prove Lemma \ref{lem:Tzeta}.

\begin{lemma}\label{lem:meanconv}
For each $t>0$, $x\in \R^3$, $\delta>0$, we have \begin{align*}
\sum_{i=2}^{\frac{t}{a_\e}}\mathbb{P}\left(\sigma^{\e,x}_i>\delta\right)\to {\frac{t}{\sqrt{2\delta}}}
\end{align*}
as $\e\to0$.
\end{lemma}

\begin{lemma}\label{lem:varconv}
For each $t>0$, $x\in \R^3$,  $\delta>0$, we have \begin{align*}
\mathbb{E}\left[\left(\sum_{i=2}^{\frac{t}{a_\e}}\left(\mathbb{P}\left(\left.\sigma^{\e,x}_i>\delta\right|\mathcal{F}_{i-1}\right)-\mathbb{P}\left(\sigma^{\e,x}_i>\delta\right) \right)\right)^2\right]\to 0
\end{align*}
as $\e\to 0$.
\end{lemma}

\begin{lemma}\label{lem:maxconv0}
For each $t>0$, $x\in \R^3$,  $\delta>0$, we have \begin{align*}
\sum_{i=2}^{\frac{t}{a_\e}}\mathbb{E}\left[\mathbb{P}\left(\left.\sigma^{\e,x}_i>\delta\right|\mathcal{F}_{i-1}\right)^2\right]\to 0.
\end{align*}
as $\e\to0$.
\end{lemma}

\begin{lemma}\label{lem:smallissmall}
For each $t>0$, $x\in \R^3$, we have \begin{align*}
\varlimsup_{\delta\to 0}\varlimsup_{\e\to 0}\mathbb{E}\left[\sum_{i=1}^{\frac{t}{a_\e}}\sigma^{\e,x}_i1\{\sigma^{\e,x}_i<\delta\}\right]=0.
\end{align*}
\end{lemma}

\begin{proof}[Proof of Lemma \ref{lem:meanconv}]
We remark that \begin{align*}
\mathbb{P}\left(\left.\sigma^{\e,x}_i>\delta\right|\mathcal{F}_{i-1}\right)=\int_{\delta}^\infty \frac{1}{h^\e\left({S_{i-1}^{\e,x}}\right)}p_{t+\e^2}\left(S_{i-1}^{\e,x}\right)\dd t=\frac{H(\delta+\e^2,S_{i-1}^{\e,x})}{H(\e^2,S_{i-1}^{\e,x})},\quad \mathbb{P}\text{-a.s.}
\end{align*}
Hence, Lemma \ref{lem:MScorr} gives \begin{align*}
\mathbb{P}\left(\sigma^{\e,x}_i>\delta\right)&=\mathbb{E}\left[\mathtt{E}_{S_{i-1}^{\e,x}}\left[\frac{H(\delta+\e^2,M_{i-2}^{\e})}{H(\e^2,M_{i-2}^{\e})}\right]\right]\\
&=\mathbb{E}\left[\mathtt{E}_{S_{i-1}^{\e,x}}\left[\frac{H(\delta+\e^2,M_{i-2}^{\e})}{H(\e^2,M_{i-2}^{\e})}\right];|S_1^{\e,x}|\leq \sqrt{\e}\right]+R_1(\e,x),
\end{align*}
where $R_{1}(\e,x)$ is the remainder term which is bounded by $C\max\{|x|,\e\}e^{-\frac{1}{4\e}}$ by Lemma \ref{cor:Sstay}.

From Lemma \ref{lem:MScorr} and \eqref{eq:Hprop4},
\begin{align*}
\mathbb{E}\left[\mathtt{E}_{S_{1}^{\e,x}}\left[\frac{H(\delta+\e^2,M_{i-2}^{\e,x})}{H(\e^2,M_{i-2}^{\e,x})}\right];|S_{1}^{\e,x}|\leq  \sqrt{\e}\right]&=\mathbb{E}\left[\mathtt{E}_{\frac{S_{1}^{\e,x}}{\e}}\left[\frac{H(\delta+\e^2,\e M_{i-2}^{1})}{H(\e^2,\e M_{i-2}^{1})}\right];{|S_{1}^{\e,x}|}\leq  \sqrt{\e}\right]\\
&=\mathbb{E}\left[\mathtt{E}_{\frac{S_{1}^{\e,x}}{\e}}\left[\frac{H(\delta\e^{-2}+1, M_{i-2}^{1})}{H(1, M_{i-2}^{1})}\right];{|S_{1}^{\e,x}|}\leq  \sqrt{\e}\right].
\end{align*} 
By $V$-uniform ergodicity of $\{M_n\}$ (Lemma \ref{lem:VuniGeoMe}), we have \begin{align*}
&\mathbb{E}\left[\mathtt{E}_{\frac{S_{1}^{\e,x}}{\e}}\left[\frac{H(\delta\e^{-2}+1, M_{i-2}^{1})}{H(1, M_{i-2}^{1})}\right];{|S_{1}^{\e,x}|}\leq  \sqrt{\e}\right]\\
&=\left(\int_{\R^3}\pi^1(\dd x) \frac{H(\delta\e^{-2}+1, x)}{H(1, x)}+R_{i-2,\e}\right)\mathbb{P}\left({|S_{1}^{\e,x}|}\leq  \sqrt{\e}\right),
\end{align*} 
where $\pi^1$ is the invariant distribution of $M^1$ given by \eqref{eq:invariantdistribution} and the error term satisfies $|R_{i-2,\e}|\leq Lr^{i-2} \e $ for any $i\geq 2$, where $L>0$ and $r\in (0,1)$ are constants independent of $\e$ and $x$. In particular,  \begin{align*}
\int_{\R^3}\pi^1(\dd x) \frac{H(\delta\e^{-2}+1, x)}{H(1, x)}=2\pi^{\frac{3}{2}}H(\delta\e^{-2}+2,0)=\frac{2\pi^{\frac{3}{2}}}{\sqrt{\delta\e^{-2}+2}}H(1,0).
\end{align*}

Also, we can see that for $i\leq (-\log \e)^2$, \begin{align*}
\mathbb{P}\left(\sigma^{\e,x}_i>\delta\right)=\mathbb{E}\left[\mathtt{E}_{S_{i-1}^{\e,x}}\left[\frac{H(\delta+\e^2,M_{i-2}^{\e})}{H(\e^2,M_{i-2}^{\e})}\right];A_\e(i-1,\e^{\frac{3}{4}})^c\right]+R_2(\e,x,i),
\end{align*}
where $R_2(\e,x,i)$ is the remainder term which is bounded by $Ci\max\{|x|,\e\}\exp\left(-{4\e^{-\frac{1}{2}}}\right)$ by Lemma \ref{cor:Sstay}. 

Proposition  \ref{prop:hprop} yields that conditioned on $A_\e(i-1,\e^{\frac{3}{4}})^c$, 
\begin{align}
\mathtt{E}_{\frac{S_{1}^{\e,x}}{\e}}\left[\frac{H(\delta\e^{-2}+1, M_{i-2}^{1})}{H(1, M_{i-2}^{1})}\right]\leq C{\int_{0}^{\frac{\e^{\frac{3}{2}}}{2( \delta+\e^2)}}u^{-\frac{1}{2}}e^{-u}\dd u}\leq C\e^{\frac{3}{4}}.\label{eq:Hdeltabdd}
\end{align}

Putting things together, we have \begin{align*}
\sum_{i=2}^{\frac{t}{a_\e}}\mathbb{P}\left(\sigma^{\e,x}_i>\delta\right)&=\sum_{i=2}^{(-\log \e)^2}\mathbb{P}\left(\sigma^{\e,x}_i>\delta\right)+\sum_{i\geq (-\log \e)^2+1}^{\frac{t}{a_\e}}\mathbb{P}\left(\sigma^{\e,x}_i>\delta\right)\\
&=o(1)+\sum_{i=(-\log \e)^2+1}^{\frac{t}{a_\e}}\frac{2\pi^{\frac{3}{2}}}{\sqrt{\delta\e^{-2}+2}}H(1,0)\left(1-o(1)\right)\to \frac{t}{\sqrt{2\delta}}
\end{align*}
\end{proof}

Modifying the proof of Lemma \ref{lem:meanconv}, we can prove Lemma \ref{lem:maxconv0}.

\begin{proof}[Proof of Lemma \ref{lem:maxconv0}] 
By the same argument as in the proof of Lemma \ref{lem:meanconv},  \begin{align*}
\mathbb{E}\left[\mathbb{P}\left(\sigma_i^{\e,x}>\delta\left|\mathcal{F}_{i-1}\right.\right)^2\right]\leq \mathbb{E}\left[\mathbb{P}\left(\sigma_i^{\e,x}>\delta\left|\mathcal{F}_{i-1}\right.\right)^2; A_\e(x,\e^{\frac{3}{4}})^c\right]+\mathbb{P}\left(A_\e(x,\e^{\frac{3}{4}})\right).
\end{align*}
By \eqref{eq:Hdeltabdd},  \begin{align}
\mathbb{E}\left[\mathbb{P}\left(\sigma_i^{\e,x}>\delta\left|\mathcal{F}_{i-1}\right.\right)^2; A_\e(x,\e^{\frac{3}{4}})^c\right]\leq C\e^{\frac{3}{2}}.\label{eq:momentdeltabdd}
\end{align}
Combining this with Lemma \ref{cor:Sstay}, we have \begin{align}
\mathbb{E}\left[\mathbb{P}\left(\sigma_i^{\e,x}>\delta\left|\mathcal{F}_{i-1}\right.\right)^2\right]\leq C\e^{\frac{3}{2}}\label{eq:2ndmomentdelta}
\end{align}
and  hence the proof is completed.
\end{proof}

\begin{proof}[Proof of Lemma \ref{lem:varconv}]
We write $X_i=\mathbb{P}\left(\sigma^{\e,x}_i>\delta\left|\mathcal{F}_{i-1}\right.\right)$ and \begin{align*}
\left(\sum_{i=2}^{\frac{t}{a_\e}}X_i-\mathbb{E}\left[X_i\right]\right)^2&=\sum_{\begin{smallmatrix}(-\log \e)^2\leq i<j\leq \frac{t}{a_\e}\\ j-i\leq (-\log \e)^2\end{smallmatrix}}\left(X_i-\mathbb{E}\left[X_i\right]\right)\left(X_j-\mathbb{E}\left[X_j\right]\right)\\
&+\sum_{\begin{smallmatrix}(-\log \e)^2\leq i<j\leq \frac{t}{a_\e}\\ j-i\geq (-\log \e)^2+1\end{smallmatrix}}\left(X_i-\mathbb{E}\left[X_i\right]\right)\left(X_j-\mathbb{E}\left[X_j\right]\right)\\
&+\sum_{\begin{smallmatrix}2\leq i\leq (-\log \e)^2 \\ i<j\leq \frac{t}{a_\e}\end{smallmatrix}}\left(X_i-\mathbb{E}\left[X_i\right]\right)\left(X_j-\mathbb{E}\left[X_j\right]\right)\\
&=:I_1^{\e}+I_2^\e+I^\e_3.
\end{align*}
By \eqref{eq:2ndmomentdelta} and H\"older's inequality, we can see that \begin{align*}
\mathbb{E}\left[I_1^\e\right]\leq \frac{t}{a_\e}C\e^{\frac{3}{2}}(-\log \e)^2 \to 0\\
\mathbb{E}\left[I_3^\e\right]\leq \frac{t}{a_\e}C\e^{\frac{3}{2}}(-\log \e)^2 \to 0
\end{align*}
as $\e\to 0$.

Thus, we focus on $\mathbb{E}[I^\e_2]$. Modifying the proof of Lemma \ref{lem:meanconv}, we can find that \begin{align*}
\sum_{\begin{smallmatrix}(-\log \e)^2\leq i<j\leq \frac{t}{a_\e}\\ j-i\geq (-\log \e)^2+1\end{smallmatrix}}\mathbb{E}[X_i]\mathbb{E}[X_j]\to \frac{t^2}{2\delta}
\end{align*} 
as $\e\to 0$.

Also, \begin{align*}
\mathbb{E}\left[X_iX_j\right]=\mathbb{E}\left[X_iX_j; A_\e(i,\e^{\frac{3}{4}})\right]+R_3(\e,x)
\end{align*}
where the remainder term $R_3(\e,x)$ is dominated by $C i\max\{|x|,\e\} \exp\left(-\frac{1}{4\sqrt{\e}}\right)$. Then, $V$-uniform integrability of $\{M_i\}$ (Lemma \ref{lem:VuniGeoMe}) yields \begin{align*}
\mathbb{E}\left[X_iX_j\right]&=\mathbb{E}\left[X_iX_j; A_\e(i,\e^{\frac{3}{4}})\right]=\mathbb{E}\left[X_i\left(\pi^1 \left[	\frac{H(\delta\e^{-2}+1,M_0)}{H(1,M_0)}		\right]+R_4(i,j,\e)\right);A_\e(i,\e^{\frac{3}{4}})\right]\\
&=\mathbb{E}\left[X_i;A_\e(i,\e^{\frac{3}{4}})\right]\left(\frac{2\pi^{\frac{3}{2}}}{\sqrt{\delta\e^{-2}+2}}H(1,0)+R_4(i,j,\e)\right),
\end{align*}
where  $R_4(i,j,\e)$ is dominated by $L\e^{\frac{3}{2}} r^{j-i}$. Repeating a similar argument, we can see that \begin{align*}
\mathbb{E}\left[X_iX_j\right]=\left(\frac{2\pi^{\frac{3}{2}}}{\sqrt{\delta\e^{-2}+2}}H(1,0)+R_5(i,\e,x)\right)\left(\frac{2\pi^{\frac{3}{2}}}{\sqrt{\delta\e^{-2}+2}}H(1,0)+R_4(i,j,\e)\right)+R_3(\e,x),
\end{align*}
where $R_5(i,\e,x)$ is dominated by $C i\max\{|x|,\e\} \exp\left(-\frac{1}{4\sqrt{\e}}\right)$. Thus, the proof is completed.
\end{proof}

\begin{proof}[Proof of Lemma \ref{lem:smallissmall}]

It follows from Lemma \ref{lem:MScorr}  that \begin{align*}
&\mathbb{E}\left[\sum_{2\leq n\leq s/a_\e }\sigma_n^{\e,x}1_{\sigma_{n,\e}\leq a}\right]\\
&=\sum_{1\leq n\leq \frac{s}{a_\e}-1}\mathbb{E}\left[\int_{\R^3}\int_{0}^a tp_t^{h,\e}(S_n^{\e,x},y)\dd y\dd t\right]\\
&= \sum_{1\leq n\leq \frac{s}{a_\e}-1 }\mathtt{E}_x\left[\int_{0}^a\frac{ tp_{t+\e^2}(M_n^\e)}{H(\e^2,M_{n}^\e)}\dd t\right]\\
&=\sum_{1\leq n\leq (-\log \e)^2 }\mathtt{E}_x\left[\int_{0}^a\frac{ tp_{t+\e^2}(M_n^\e)}{H(\e^2,M_{n}^\e)}\dd t\right]+\sum_{(-\log \e)^2+1\leq n\leq \frac{s}{a_\e}-1 }\mathtt{E}_x\left[\int_{0}^a\frac{ tp_{t+\e^2}(M_n^\e)}{H(\e^2,M_{n}^\e)}\dd t\right].
\end{align*}
By Lemma \ref{prop:Mfar'}, we have \begin{align*}
&\sum_{1\leq n\leq (-\log \e)^2 }\mathtt{E}_x\left[\int_{0}^a\frac{ tp_{t+\e^2}(M_n^\e)}{H(\e^2,M_{n}^\e)}\dd t\right]\\
&\leq \sum_{1\leq n\leq (-\log \e)^2 }\left(\mathtt{E}\left[\int_{0}^a\frac{ tp_{t+\e^2}(M_n^\e)}{H(\e^2,M_{n}^\e)}\dd t;A_\e(n,\e^{\frac{3}{4}})\right]+Ca\mathbb{P}(A_\e(n,\e^\frac{3}{4})^c)\right)\\
&\leq \sum_{1\leq n\leq (-\log \e)^2 }\left(C\e^{\frac{3}{4}}\int_0^a \frac{t}{(2\pi (t+\e^2))^\frac{3}{2}}\dd t+Ca\mathbb{P}(A_\e(n,\e^\frac{3}{4})^c)\right)\to 0
\end{align*}
as $\e\to 0$, where we have used \eqref{eq:Hprop5} in the last line.

Also, by  Lemma \ref{lem:scalingMCM} and \eqref{eq:Hprop5}\begin{align*}
\sum_{(-\log \e)^2+1\leq n\leq \frac{s}{a_\e}-1 }\mathtt{E}_x\left[\int_{0}^a\frac{ tp_{t+\e^2}(M_n^\e)}{H(\e^2,M_{n}^\e)}\dd t\right]=\sum_{(-\log \e)^2+1\leq n\leq \frac{s}{a_\e}-1 }\mathtt{E}_{\frac{x}{\e}}\left[\int_{0}^a\frac{ tp_{t\e^{-2}+1}(M_n^1)}{\e^3 H(1,M_{n}^1)}\dd t\right].
\end{align*}
We can use  the $V$-uniform ergodicity of $\{M^1_n\}$ so that \begin{align*}
\mathtt{E}_{\frac{x}{\e}}\left[\int_{0}^a\frac{ tp_{t\e^{-2}+1}(M_n^1)}{\e^3 H(1,M_{n}^1)}\dd t\right]\leq \int_{\R^3}\int_{0}^a\frac{ tp_{t\e^{-2}+1}(x)}{\e^3 H(1,x)}\dd t\pi^1								(\dd x)+Lr^n\frac{|x|}{\e^4}\leq Ca+Lr^n\frac{|x|}{\e^4}
\end{align*}
and hence find that \begin{align*}
\varlimsup_{\e\to 0}\sum_{1\leq n\leq  s/a_\e-1  }\mathtt{E}_x\left[\int_{0}^a\frac{ tp_{t+\e^2}(M_n^\e)}{H(\e^2,M_{n}^\e)}\dd t\right]\leq Ca.
\end{align*}

\end{proof}



\subsection{Convergence of $W^{\e,x}$}\label{subsec:Convomegaex}

Here, we will prove the weak convergence of $\{W^{\e,x}\}$ in $C([0,T],\R^3)$ with the uniform topology  for any $T\geq 0$.  

It is enough to show that the weak convergence of subsequence $\{\e_n\}_{n\geq 1}$ with $\e_n\searrow 0 $. 

For fixed $T>0$, we set $t_T:=\inf \{t>T;t\in R_{\widetilde{T}^x}\}$. Then, we can label the intervals of $[T^x,t_T]\backslash R_{\widetilde{T}^x}$ in  decreasing order of their lengths: $\{I_n=I_n(T)\}_{n\geq 1}$. 
\vskip\baselineskip
Now, we construct the limit process from $R_{\widetilde{T}^x}$ as follows.

Let $\{B(t):t\geq 0\}$ and $\{B^{\mathrm{Bri}}_n(t):0\leq t\leq 1\}_{n\geq 1}$ be a Brownian motion starting at $0$ and an i.i.d.~sequence of standard Brownian bridges in $\R^3$ which are independent of each other, respectively. For  $m\in \mathbb{N}$, let $N_m(T)=N_m=\sup\{n\geq 1;|I_n|>\frac{1}{m}\}$. Now, we define $C([0,\infty),\R^3)$-valued process $\mathcal{W}_m$ by \begin{align*}
\mathcal{W}_m(s)=\begin{cases}
\dis x\left(1-\frac{s}{\sigma^x}\right)+\sqrt{\sigma^x}B\left(\frac{s}{\sigma^x}\right),\quad &s\in [0,\sigma^x]\\
\dis \sqrt{|I_n|}B^{\mathrm{Bri}}_{n}\left(\frac{s-\inf I_n}{|I_n|}\right),\quad &s\in I_n \text{ for }n\leq N_m\\
\dis 0,\quad &s\in I_n \text{ for }n\geq N_m+1 \\
\dis 0,\quad &s\in R_{\widetilde{T}^x\backslash \{0\}}\\
\dis 0,\quad &s\geq t_T.
\end{cases}
\end{align*}

\begin{theorem}\label{thm:limitprocess}
Fix $T>0$ and $x\in \R^3\backslash \{0\}$. Then,  \begin{align*}
\{\mathcal{W}_m\}_{m\geq 1} \text{is a $C([0,\infty),\R^3)$-Cauchy sequence, $\mathbb{P}(\cdot|\widetilde{T}^x)$-a.s.}
\end{align*} 
and hence there exists a process $\left\{X^{x,T}(s)\right\}_{s\geq 0 }$ such that $\mathcal{W}_m\to X^{x,T}$ in the locally uniform convergence sense $\mathbb{P}$-a.s. 

Moreover, there exists a random process $\{X^x(s)\}_{s\geq 0}$ such that for any $T>0$ the restriction of $X^x$ to $[0,T]$ has the same distribution as $X^{x,T}$.
\end{theorem}

\begin{proof}
It is clear that for each $m\geq \ell\geq 1$, \begin{align*}
\mathcal{W}_m(s)-\mathcal{W}_\ell(s)=\begin{cases}
\dis \sqrt{|I_n|}B^{\mathrm{Bri}}_{n}\left(\frac{s-\inf I_n}{|I_n|}\right)\quad &s\in I_n \text{ for }N_\ell <n\leq N_m\\
0,\quad &\text{otherwise}
\end{cases}
\end{align*}
Hence, we can see that for each $\e>0$, \begin{align*}
\mathbb{P}\left(\left|\mathcal{W}_m(s)-\mathcal{W}_\ell(s)\right|>\e\right)\leq \sum_{j=\ell+1}^m \mathbb{P}\left(\bigcup_{N_j<n\leq N_{j+1} }\left\{	\sup_{s\in [0,1]}\left|B^{\mathrm{Bri}}_n(s)\right|		>\e\sqrt{j}\right\}\right).
\end{align*}
Since $\sharp (N_{j+1}\backslash N_j)\leq T(j+1)$ and the  distribution of  maximum of standard Brownian bridge in $\R$ is given by \begin{align}
P\left(\sup_{0\leq s\leq t}B_s^{\{\mathrm{Bri},1\}}\geq u\right)=e^{-2u^2}
\end{align}
for $u>0$ (see \cite[(1)]{PY99} for instance),  \begin{align*}
\mathbb{P}\left(\bigcup_{N_j<n\leq N_{j+1} }\left\{	\sup_{s\in [0,1]}\left|B^{\mathrm{Bri}}_n(s)\right|		>\e\sqrt{j}\right\}\right)\leq 6T(j+1)e^{-2j\e^2}.
\end{align*}
Thus, the first statement follows from Borel-Cantelli's lemma.

The second statement follows from the $C([0,\infty),\R^3)$-tightness and finite dimensional convergence both of which comes from the first statement.

\end{proof}

Finally, we prove that $W^{\e,x}$ converges to $X^x$.

\begin{proof}

Let $\widetilde{W}^{\e,x}$ be a process constructed in the same way as $W^{\e,x}$ for the renewal points $\{(T_n^{\e,x},0)\}_{1\leq n\leq \e^{-2}}$.  

Then, Proposition \ref{prop:Mfar'} implies  that \begin{align*}
\mathbb{P}\left(\left|\widetilde{W}^{\e,x}-W^{\e,x}\right|>\frac{1}{m}\right)\leq \frac{1}{\e^2}\mathbb{P}\left(A_\e\left(\e^{-2},\frac{1}{m\e}\right)^c\right)\to 0
\end{align*}
for each $m\geq 1$.

As we discussed before Corollary \ref{cor:SkotoMat}, we can retake another probability space denoting by $(\Omega,\mathcal{F},\mathbb{P})$  such that $\widetilde{T}^{\e,x}$ converges to  $\widetilde{T}^x$, $\mathbb{P}$-a.s. Thus,  Lemma \ref{lem:Cconveuqiv} and Lemma \ref{lem:setcounvergence} imply that there exists a sequence of pairs $\{(s_n^{(j)},t_n^{(j)}):j=1,\dots,\sharp N_m(T)\}$ for $m\geq 1$ and $T>0$ such that \begin{align*}
s_n^{(j)},t_n^{(j)}\in \{\eta^{\e_n,x}\}\quad \text{and }\lim_{n\to \infty}s_n^{(j)}=s^{(j)}, \lim_{n\to\infty}t_n^{(j)}=t^{(j)},
\end{align*}
for each $j\geq 1$,  $\mathbb{P}$-a.s. In particular, we may retake $\{(s_n^{(j)},t_n^{(j)})\}$ such that \begin{align*}
&t_n^{(j)}-s_n^{(j)}>\frac{1}{m},\quad \text{and }(s_n^{(j)},t_n^{(j)})\cap \{\eta^{\e_n,x}\}=\emptyset \text{ for large $n$}
\intertext{and}
&T_{k}^{\e,x}-T_{k-1}^{\e,x}<\frac{1}{m}\quad \text{for }(T_{k-1}^{\e,x},T_k^{\e,x})\not=(s_n^{(j)},t_n^{(j)}), j=1,\dots,N_m(T), \text{ for large $n$.} 
\end{align*}

We can reconstruct $\widetilde{W}^{\e,x}$ by attaching Brownian bridges $\{B^{\mathrm{Bri}}_n\}$ by decreasing order of lengths of intervals and we set \begin{align*}
W^{(\e,x,m,T)}(s)=\begin{cases}
\dis x\left(1-\frac{s}{T_1^{\e,x}}\right)+\sqrt{T_1^{\e,x}}B\left(\frac{s}{T_1^{\e,x}}\right),\quad &s\in [0,T_1^{\e,x}]\\
\dis \sqrt{t_n^{(j)}-s_n^{(j)}}B^{\mathrm{Bri}}_{j}\left(\frac{s-s_n^{(j)}}{t_n^{(j)}-s_n^{(j)}}\right),\quad &s\in [s_n^{(j)},t_n^{(j)}] \text{ for }j=1,\dots,N_m(T)\\
\dis 0,\quad &\textrm{otherwise}.
\end{cases}
\end{align*}

Then, we can use the same argument as in the proof of Theorem \ref{thm:limitprocess} so that for any $\delta>0$\begin{align*}
\mathbb{P}\left(\sup_{0\leq s\leq T}\left|\widetilde{W}^{\e,x}(s)-{W}^{\e,x,m,T}(s)\right|>\delta\right)\leq \sum_{j=m}^\infty 6T(j+1)e^{-2T\delta^2}.
\end{align*}

Finally, we can find from the continuity of Brownian bridges and the construction of $W^{(\e,x,m,T)}$ that \begin{align*}
W^{(\e,x,m,T)}\to \mathcal{W}_m
\end{align*}
uniformly in $s\in [0,T]$, $\mathbb{P}$-a.s. Therefore, we can concluded the statement.

\end{proof}

\begin{rem}
We remark that the strategy of the proof is as follows:
\begin{enumerate}
\item Convergence of ``renewal points".
\item Processes by connecting renewal points by Brownian bridges weakly converges in locally uniform convergence topology.
\end{enumerate}
In particular, the argument in the second part is independent of the limit of renewal points. 
\end{rem}

\section{Limit of $\eta^{\e,t,x}$}
Our goal is to show that for $u_0\in C_c(\R^3\backslash \{0\})$ and $(t,x)\in (0,\infty)\times \R^3$, $u^\e(t,x)$ converges to a  non-trivial function $u(t,x)$ and give a Feynman-Kac formula for $u(t,x)$ in terms of the scaling limit of $({U}^{\e,t,x},W^{\e,t,x})$.

In this section, we study the scaling limit of $\left\{W^{\e,t,x}(s)\right\}_{0\leq s\leq t}$.

\begin{lemma}\label{lem:GetoG}
For any $t>0$ and $x\in \R^3\backslash \{0\}$, $G_{\e,t}(x)$ converges to 
\begin{align*}
G_t^\gamma(x)= 1+\sqrt{\pi}\int_0^\infty \dd s e^{s \gamma} \int_{0}^t p_{u}(x)P(\zeta_s\leq t-u)\dd u.
\end{align*}

\end{lemma}

\begin{proof}
By definition, \begin{align*}
G_{\e,t}(x)=1+\sum_{n\geq 1}\lambda(\e)^nh^\e(x)\mathbb{E}\left[\frac{1\{T_{n}^{\e,x}\leq t\}}{h^\e(S_n^{\e,x})}\right].
\end{align*}
So, we focus on the summation. Without loss of generality, we may assume $|x|\geq \e$.

First,  we remark from Lemma \ref{prop:Mfar'} that for any $n\geq 1$, we have \begin{align*}
\mathbb{E}\left[\frac{1}{h^{\e}(S_n^{\e,x})^2}\right]&\leq C\mathtt{E}_x\left[\left|M_n^{\e}\right|^2\right]+C\e^2\\
&=C\int_0^\infty R\mathtt{P}_x\left(\left|M_n^\e\right|\geq R\right)\dd R+C\e^2\\
&\leq Cn\int_0^\infty R\frac{\left({\max\{|x|, \e\}}+\max\{R,\e\}\right)}{\e}\frac{\max\{R,\e\}}{\e}{e^{-\frac{R^2}{2\e^2}}}  \dd R+C\e^2\\
&\leq Cn\left(\e\max\{|x|,\e\}+\e^2\right)+C\e^2.
\end{align*}

Fix $0<\delta<T<\infty$. Then we consider the decomposition
\begin{align*}
\sum_{n\geq 1}\lambda(\e)^nh^\e(x)\mathbb{E}\left[\frac{1\{T_{n}^{\e,x}\leq t\}}{h^\e(S_n^{\e,x})}\right]
&=\sum_{n=1}^{\frac{\delta}{a_\e} }\lambda(\e)^nh^\e(x)\mathbb{E}\left[\frac{1\{T_{n}^{\e,x}\leq t\}}{h^\e(S_n^{\e,x})}\right]\\
&+\sum_{n=\frac{\delta}{a_\e}+1}^{{\frac{T}{a_\e}}}\lambda(\e)^nh^\e(x)\mathbb{E}\left[\frac{1\{T_{n}^{\e,x}\leq t\}}{h^\e(S_n^{\e,x})}\right]\\
&+\sum_{n\geq {\frac{T}{a_\e}}+1}\lambda(\e)^nh^\e(x)\mathbb{E}\left[\frac{1\{T_{n}^{\e,x}\leq t\}}{h^\e(S_n^{\e,x})}\right]\\
&=:I_{\e,t,x}^{(1)}+I_{\e,t,x}^{(2)}+I_{\e,t,x}^{(3)}. 
\end{align*}
We will complete the proof by looking at the convergence of each term.
\end{proof}

\begin{lemma}\label{lem:I1estimate}
There exists a constant $C>0$ such that \begin{align*}
I_{\e,t,x}^{(1)}\leq C\e \frac{\delta}{a_\e}
\end{align*}
for uniformly in $|x|\geq r$.
\end{lemma}

\begin{proof}
First, we note that \begin{align*}
\lambda(\e)^\frac{s}{a_\e}=\left(1+\gamma{a_\e}\right)^\frac{t}{a_\e}\sim e^{\gamma t}
\end{align*}
as $\e\to 0$ for each $t\geq 0$. Also, for each $n\geq 0$,  
\begin{align*}
\lambda(\e)^n\leq e^{\frac{|\gamma|n}{a_\e}}.
\end{align*}

Thus, we obtain from Lemma \ref{lem:MScorr} that \begin{align*}
I_{\e,t,x}^{(1)}&\leq \sum_{n=1}^{\frac{\delta}{a_\e} }\lambda(\e)^nh^\e(x)\mathbb{E}\left[\frac{1}{h^\e(S_n^{\e,x})}\right]\leq Ch^\e(x)\sum_{n=1}^{\frac{\delta}{a_\e} }\mathtt{E}_x\left[\frac{1}{h^\e(M_n^{\e})}\right].
\end{align*}

Lemma \ref{lem:scalingMCM} and the Markovo proprty imply that 
\begin{align*}
\mathtt{E}_x\left[\frac{1}{h^\e(M_n^{\e})}\right]&=\e\mathtt{E}_{\frac{x}{\e}}\left[\frac{1}{h(M_n^{1})}\right]\\
&=\e \mathtt{E}_{\frac{x}{\e}}\left[\mathtt{E}_{M_1^1}\left[\frac{1}{h(M_{n-1}^{1})}\right]\right].
\end{align*}

By \eqref{eq:Hprop5}, we can apply Lemma \ref{lem:VuniGeoMe} and we can see that  there exists a constant $L>0$ such that \begin{align*}
\mathtt{E}_{M_1^1}\left[\frac{1}{h(M_{n-1}^{1})}\right]&= \int_{\R^3} \frac{1}{h(y)}\pi^1(\dd y)+R_n\\
&=2\pi^\frac{3}{2}+R_n,
\end{align*}
where where $\pi^1$ is given by \eqref{eq:invariantdistribution} and $R_n$ is a random variable with  
\begin{align*}
\left|R_n\right|\leq L\max\{|M_1^1|,1\} r^{n-1}.
\end{align*}

Also, we can see that for any $|x|\geq \e>0$ \begin{align*}
\mathtt{E}_{\frac{x}{\e}}\left[\max\{|M_1^1|,1\}\right]&\leq C\int_{\R^3}\frac{|y|p_1(y)}{\left|\frac{x}{\e}-y\right|h^1\left(\frac{x}{\e}\right)}+1\\
&\leq C\frac{|x|}{\e}\int_{\R^3}\frac{|y|p_1(y)}{\left|\frac{x}{\e}-y\right|}\dd y+1.
\end{align*}
Since we have \begin{align*}
\int_{\R^3}\frac{|y|p_1(y)}{\left|\frac{x}{\e}-y\right|}\dd y&=\int_{\left|\frac{x}{\e}-y\right|\leq \frac{|x|}{2\e}}\frac{|y|p_1(y)}{\left|\frac{x}{\e}-y\right|}\dd y+\int_{\left|\frac{x}{\e}-y\right|\geq \frac{|x|}{2\e}}\frac{|y|p_1(y)}{\left|\frac{x}{\e}-y\right|}\dd y\\
&\leq C\int_{\left|\frac{x}{\e}-y\right|\leq \frac{|x|}{2\e}}\frac{\frac{|x|}{2\e}e^{-\frac{|x|^2}{8\e^2}}}{\left|\frac{x}{\e}-y\right|}\dd y+\frac{2\e}{|x|}\int_{\R^3}|y|p_1(y)\dd y,
\end{align*}
it follows that \begin{align}
\mathtt{E}_{M_1^1}\left[\frac{1}{h(M_{n-1}^{1})}\right]=2\pi^\frac{3}{2}+r^{n-1}O(1).\label{eq:h(S)exp}
\end{align} 

Therefore, we can find a constant $C>0$ such that \begin{align*}
I_{\e,t,x}^{(1)}\leq C\e \frac{\delta}{a_\e}
\end{align*}
for uniformly in $|x|\geq r$.

\end{proof}

\begin{lemma}\label{lem:I2asy} 
\begin{align*}
I_{\e,t,x}^{(2)}\to \pi^\frac{1}{2}\int_\delta^T e^{\gamma s} \left(\int_0^t p_u(x) \mathbb{P}(\tau_{s}\leq t-u)\dd u\right)\dd s.
\end{align*}
\end{lemma}

\begin{proof}
We set $b_\e=a_\e^{-\frac{1}{3}}$. 
We remark that 
\begin{align*}
&\sum_{n=\frac{\delta}{a_\e}+1}^{{\frac{T}{a_\e}}}\lambda(\e)^nh^\e(x)\mathbb{E}\left[\frac{1\{T_{n}^{\e,x}\leq t\}}{h^\e(S_n^{\e,x})}\right]\\
&=\sum_{n=\frac{\delta}{a_\e}+1}^{{\frac{T}{a_\e}}}\lambda(\e)^nh^\e(x)\left(\mathbb{E}\left[\frac{1\{T_{(n-b_\e)}^{\e,x}\leq t\}}{h^\e(S_n^{\e,x})}\right]-\mathbb{E}\left[\frac{1\{T_{(n-b_\e)}^{\e,x}\leq t, T_n^{\e,x}>t\}}{h^\e(S_n^{\e,x})}\right]\right)\\
&=:I^{(2,1)}_{\e,t,x}-I^{(2,2)}_{\e,t,x}.
\end{align*}
We can see 
from the Markov property and Lemma \ref{lem:VuniGeoMe} that \begin{align*}
\mathbb{E}\left[\frac{1\{T_{(n-b_\e)}^{\e,x}\leq t\}}{h^\e(S_n^{\e,x})}\right]&=\mathbb{E}\left[{1\{T_{(n-b_\e)}^{\e,x}\leq t\}}\mathtt{E}_{S_{n-b_\e}^{\e,x}}\left[\frac{1}{h^{\e}(M_{b_\e}^\e)}\right]\right]\\
&=\mathbb{P}\left(T_{(n-b_\e)}^{\e,x}\leq t\right)\int_{\R^3}\e \frac{\rho^1(y)}{h^1(y)}\dd y+R_1(\e,x,n,u)\\
&=2\pi^\frac{3}{2}\e\mathbb{P}\left(T_{n-b_\e}^{\e,x}\leq t\right)+R_1(\e,x,n,u)
\end{align*}
where \begin{align*}
\left|R_1(\e,x,n,u)\right|\leq Cr^{b_\e}\e \left( \mathtt{E}_{\frac{x}{\e}}\left[\left|\frac{1}{h\left(M_{b_\e}^1\right)}\right|\right]+1\right)\leq Cr^{b_\e}\e ,
\end{align*}
where we have used \eqref{eq:h(S)exp} in the last inequality.

We know from Lemma \ref{lem:Tzeta} that \begin{align*}
\mathbb{P}\left(T_{\frac{s}{a_\e}-b_\e}-T_1^{\e,x}\leq t\right)\to \mathbb{P}(\zeta_s\leq t)
\end{align*}
for $t\leq [0,\infty)$, and $s\in (\delta,T)$.

Thus, we can see by the dominated convergence theorem that \begin{align}
&\sum_{n=\frac{\delta}{a_\e}+1}^{{\frac{T}{a_\e}}}\lambda(\e)^nh^\e(x)\mathbb{E}\left[\frac{1\{T_{n-b_\e}^{\e,x}\leq t\}}{h^\e(S_n^{\e,x})}\right]\notag\\
&= h^\e(x)\int_\delta^T \lambda(\e)^{\left\lfloor\frac{s}{a_\e}\right\rfloor}\mathbb{E}\left[\frac{1\{T_{\left\lfloor\frac{s}{a_\e}\right\rfloor-b_\e}^{\e,x}\leq t\}}{h^\e\left(S_{\left\lfloor\frac{s}{a_\e}\right\rfloor}^{\e,x}\right)}\right]\frac{\dd s}{a_\e}\notag\\
&=2\pi^\frac{3}{2}\e h^\e(x)\int_\delta^T \lambda(\e)^{\left\lfloor\frac{s}{a_\e}\right\rfloor}\mathbb{P}\left(T_1^{\e,x}+\left(T_{\left\lfloor\frac{s}{a_\e}\right\rfloor-b_\e}^{\e,x}-T_{1}^{\e,x}\right)\leq t\right)\frac{\dd s}{a_\e}+o(1)\notag\\
&\rightarrow \pi^\frac{3}{2}\int_\delta^T e^{\gamma s}\frac{1}{2\pi |x|} \left(\int_0^t 2|x|p_u(x) \mathbb{P}(\tau_{s}\leq t-u)\dd u\right)\dd s.\label{eq:dominatedconv}
\end{align}

We denote by \begin{align*}
&g^\e(s):=\frac{h^\e(x) }{a_\e}\lambda(\e)^{\left\lfloor\frac{s}{a_\e}\right\rfloor}\mathbb{E}\left[\frac{1\left\{T_{\left\lfloor\frac{s}{a_\e}\right\rfloor-b_\e}^{\e,x}\leq t, T_{\left\lfloor\frac{s}{a_\e}\right\rfloor}^{\e,x}>t\right\}}{h^\e\left(S_{\left\lfloor\frac{s}{a_\e}\right\rfloor}^{\e,x}\right)}\right].
\end{align*}

By H\"older's inequality, we obtain  \begin{align*}
g^\e(s)\leq \frac{h^\e(x) }{a_\e}\lambda(\e)^{\left\lfloor\frac{s}{a_\e}\right\rfloor} \mathbb{P}\left(T_{\left\lfloor\frac{s}{a_\e}\right\rfloor-b_\e}^{\e,x}\leq t, T_{\left\lfloor\frac{s}{a_\e}\right\rfloor}^{\e,x}>t\right)^\frac{1}{2}\mathbb{E}\left[\frac{1}{h^\e\left(S_{\left\lfloor\frac{s}{a_\e}\right\rfloor}^{\e,x}\right)^2}\right]^\frac{1}{2}.
\end{align*}
A similar argument to \eqref{eq:h(S)exp} yields that \begin{align}
\mathbb{E}\left[\frac{1}{h^\e\left(S_{n}^{\e,x}\right)^2}\right]^\frac{1}{2}\leq C\e\label{eq:h(s)2exp}
\end{align}
Also, we can find from Lemma \ref{lem:Tzeta}  that \begin{align*}
\mathbb{P}\left(T_{\left\lfloor\frac{s}{a_\e}\right\rfloor-b_\e}^{\e,x}\leq t, T_{\left\lfloor\frac{s}{a_\e}\right\rfloor}^{\e,x}>t\right)\to 0.
\end{align*}
Therefore,  $g^\e(s)\to 0$  as $\e\to 0$ and $g^\e(s)$ is bounded for $\e>0$ in $[\delta,T]$. Hence, it follows from  bounded convergence theorem that \begin{align*}
\lim_{\e\to 0}\int_\delta^T g^\e(s)\dd s=0.
\end{align*}
\end{proof}

\begin{lemma}\label{lem:I3}
\begin{align*}
\varlimsup_{T\to\infty} \varlimsup_{\e\to 0}I_{\e,t,x}^{(3)}=0.
\end{align*}
\end{lemma}

\begin{proof}
 \eqref{eq:h(s)2exp} and Corollary \ref{cor:Tdeciation} imply that  \begin{align*}
\mathbb{E}\left[\frac{1\{T_{n}^{\e,x}\leq t\}}{h^\e(S_n^{\e,x})}\right]\leq 
\mathbb{P}\left(T_{n}^{\e,x}\leq t\right)^\frac{1}{2} \mathbb{E}\left[\frac{1}{h^\e(S_n^{\e,x})^2}\right]^\frac{1}{2}\leq C\e \left(1-\frac{\lambda^\frac{1}{4}}{2}\e\right)^{\frac{n}{2}}.
\end{align*}
Taking $\lambda> 16|\gamma|^4$ large, we can find from  the same argument as in Lemma \ref{lem:I2asy} that \begin{align*}
\varlimsup_{\e\to 0}I_{\e,t,x}^{(3)}\leq C e^{-\tilde{\delta} T} 
\end{align*}
for some $\tilde{\delta}>0$.

\end{proof}

Refining the proof of Lemma \ref{lem:Tlowerbdd}, we can show the following:
\begin{lemma}\label{lem:Feynman}
Fix $t>0$ and $x\in\R^3\backslash \{0\}$. Then, \begin{align*}
u^\e(t,x)\to S_tu_0(x)+\sqrt{\pi}\int_{0}^{\infty}e^{\gamma v}\dd v \int_{0<u<s<t}\dd up_u(x)\mathbb{P}(\zeta_v\in \dd s)S_{t-s}u_0(0)
\end{align*}
for $u_0\in C_b\left(\R^3\right)$, where $S_tu_0(x)=\int_{\R^3}p_t(x,y)u_0(y)\dd y$.

\end{lemma}

\section{Convergence of processes}


In this section, we will show the convergence of the processes $\{W^{\e,t,x}(s)\}_{0\leq s\leq t}$. 

Lemma \ref{lem:GetoG} implies \begin{align*}
P_{\e,t,x}\left({U}^{\e,t,x}=0\right)=\frac{1}{G_{\e,t}(x)}\to \frac{1}{G_t^\gamma(x)}
\end{align*}
for $t>0$ and $x\in \mathbb{R}^3\backslash \{0\}$. In particular, the conditional law of $W^{\e,t,x}$ on $\{{U}^{\e,t,x}=0\}$ corresponds  with the one of three dimensional Brownian motion. Thus, we have the following:
\begin{lemma}\label{lem:U0limit}
For any $A\in \mathcal{B}\left(C([0,t],\R^3)\right)$, we have \begin{align*}
\mathbb{P}_{\e,t,x}\left(W^{\e,t,x}_{\cdot}\in A,{U}^{\e,t,x}=0\right)\to \frac{1}{G^\gamma_t(x)}\mu(A).
\end{align*}
\end{lemma}

Therefore, we focus on the conditional law $\mathbb{P}^{\e,t,x}\left(W^{\e,t,x}\in A\left|\mathcal{U}^{\e,t,x}\geq 1\right.\right)$.

By the proof of Lemma \ref{lem:GetoG}, we can see that \begin{align*}
\frac{1}{a_\e}\mathbb{P}^{\e,t,x}\left({U}^{\e,t,x}=\frac{s}{a_\e}\right)=\frac{h^\e(x)}{G_{\e,t}(x)}\frac{1}{a_\e}\mathbb{E}\left[\frac{1\{T_{\frac{s}{a_\e}}\leq t\}}{h^\e(S_\frac{s}{a_\e})}\right]\to \frac{1}{G_t^\gamma(x)}\sqrt{\pi }{e^{\gamma s}}\int_0^t p_u(x)\mathbb{P}(\zeta_s\leq t-u)\dd u.
\end{align*}

Conditioned on $\left\{{U}^{\e,t,x}=\frac{s}{a_\e}\right\}$, we define an associated process $\widetilde{T}^{\e,t,x}$ in a similar way to $\widetilde{T}^{\e,x}$: \begin{align*}
\widetilde{T}^{\e,t,x}(u)=\begin{cases}
0,\quad &u\in [0,1)\\
T_1^{\e,t,x},\quad &u=1\\
T_n^{\e,t,x}, \quad &u\in [(n-1)a_\e+1,n a_\e+1), \text{ for }n=1,\dots,\frac{s}{a_\e}\\
\infty,\quad &u\geq s+1 .
\end{cases}
\end{align*} 
Then, we can see from  the weak convergence of the process $\widetilde{T}^{\e,x}$ (Lemma \ref{lem:Tzeta}) that \begin{align*}
\mathbb{P}\left(\left.\widetilde{T}^{\e,t,x}\in\cdot\right|{U}^{\e,t,x}=\frac{s}{a_\e}\right)\to \mathbb{P}\left(\left.\widetilde{T}^{x}\in \cdot\right|\zeta_s\leq t-\sigma^x \right).
\end{align*}
Indeed, the same argument as in the proof of Lemma \ref{lem:I2asy} yields that \begin{align*}
\mathbb{E}\left[\frac{f\left(\widetilde{T}^{\e,x}\right)1\left\{\widetilde{T}_{\frac{s}{a_\e}+1}^{\e,x}\leq t\right\}}{h^\e\left(S_{\frac{s}{a_\e}}^{\e,x}\right)}\right]\sim 2\pi^\frac{3}{2}\e \mathbb{E}\left[f\left(\widetilde{T}^{\e,x}\right)1\left\{\widetilde{T}_{\frac{s}{a_\e}+1}^{\e,x}\leq t\right\}\right]
\end{align*}
for any bounded continuous function $f$ on $\left(\mathcal{D},d_\mathrm{S}\right)$.

Thus, repeating the same argument as the convergence of $\{W^{\e,x}\}$ yields  the convergence of $W^{\e,t,x}$. In particular, the limit process is constructed as follows:
\begin{enumerate}
\item We consider a $[0,\infty)$-valued random variable $\mathcal{U}^{t,x}$ whose law is given by \begin{align*}
&\mathbb{P}\left(\mathcal{U}^{t,x}=0\right)=\frac{1}{G_t^\gamma(x)},\\
&\mathbb{P}\left(\mathcal{U}^{t,x}\in (a,b)\right)=\frac{1}{G_T^\gamma(x)}\sqrt{\pi}\int_a^b \dd s e^{s \gamma} \int_{0}^t p_{u}(x)P(\zeta_s\leq t-u)\dd u.\quad 0\leq a<b<\infty.
\end{align*}
\item Conditioned on $\{\mathcal{U}^{t,x}=s\}$, \begin{enumerate}
\item $s=0$, $W^{t,x}$ is the three dimensional Brownian motion starting at $x$.
\item $s>0$, we look at  $\widetilde{T}^{t,x}$ conditioned on $\{\widetilde{T}_{s+1}^x\leq  t\}$ and its range $R_{\widetilde{T}}$ in $[0,t]$. Then, we can construct $W^{t,x}$  by the same way as in $W^{x}$ up to time $\widetilde{T}^{x}(s+1)$ and after  $\widetilde{T}^{x}(s+1)$, $W^{t,x}$ behaves three dimensional Brownian motion. 
\end{enumerate}
\end{enumerate}

\section{Comments}

In the end of this article, we give some remarks on our Feynman-Kac formula associated with  \eqref{eq:HEonepoint}.

We shall recall that we prepared the process $\omega^{t,x}$ with the law $\mu_{t,x} $ and the functions $G^{(\alpha)}_t$ to describe the solution of \eqref{eq:HEonepoint}. Comparing with the standard Feynman-Kac formula \eqref{eq:FKsta}, the normalizing constant $G^{(\alpha)}_t$ seems to be unnatural. 

To avoid $G^{(\alpha)}_t$, it would be better to see the Feynman-Kac formula via super-Brownian motion introduced by \cite{Wat68,Daw75}.

Let $M_F(\R^d)$ be a set of finite measure on $(\R^d,\mathcal{B}(\R^d))$ with the topology of weak convergence.   We denote $\mu(f)=\int f(x)\mu(\dd x)$ for  $f\in C_b(\R^d)$ and $\mu\in M_F(\R^d)$.  A super-Brownian motion $\{X_t\}_{t\geq0 }$ is an $M_F(E)$-valued process which is characterized by $(-\frac{1}{2}\Delta,\alpha,\beta,\gamma,D)$ and their initial measures, where $\alpha\in (0,1]$ is a constant, $\beta$ and $\gamma$ are bounded continuous functions on $\R^d$  with $\gamma(x)\geq 0$ for $x\in\R^d$, and $D\subset \R^d$ is a domain.

Then, it is known that for $u_0\in C_b(\R^d)$, 
\begin{align*}
\mathtt{E}_{\delta_x}\left[X_t(u_0)\right]=E_x\left[\exp\left(\int_0^t \beta(\omega_s)\dd s\right)u_0(\omega_t)\right],
\end{align*}
where we denote by $\mathtt{P}_{\mu}$ and $\mathtt{E}_\mu$ the law and the expectation of super-Brownian motion with initial condition $\mu$ \cite[Exercise II.5.2(b)]{Per02}.

If we consider a family of super-Brownian motions $\{X^\e\}$ with $\left(-\frac{1}{2}\Delta,-\lambda(\e) V^\e,1,\R^3\right)$, then \begin{align*}
\mathtt{E}_{\delta_x}\left[X^\e_t(u_0)\right]=E_x\left[\exp\left(-\int_0^t \lambda(\e)V_\e(\omega_s)\dd s\right)u_0(\omega_t)\right].
\end{align*}
By Lemma \ref{lem:Feynman}, we know that the right-hand side converges to the non-trivial limit for $x\not=0$, and hence we may believe that there exists a  super-Brownian motion $\{\mathcal{X}_t\}_{t\geq 0}$ such that for $u_0\in C_c(\R^3\backslash \{0\})$, \begin{align*}
\mathtt{E}_{\delta_x}[\mathcal{X}_t(u_0)]=G_t(x)\mathbb{E}\left[u_0\left(W^{t,x}(t)\right)\right]=u_\alpha(t,x)
\end{align*}
for any $t>0$ and any $x\in \R^3\backslash\{0\}$.

Indeed, Fleischman and Mueller constructed the super-Brownian motion $\{X_t\}_{t\geq 0}$ with log-Laplace transition functional satisfies \begin{align*}
&\frac{\partial}{\partial_t} v=\frac{1}{2}\Delta^\alpha v-\eta v^{1+\beta}\quad \text{on }(0,\infty)\times \R^d\backslash\{0\}\\
&\lim_{t\to +0}v(t,x)=v(x)\geq 0\quad x\in \R^d\backslash\{0\}
\end{align*}
for $\beta\in (0,1]$ when $d=2$ and $\beta\in (0,1)$ when $d=3$.


Their result may suggest a physical perspective of super-Brownian motions via Schr\"odinger operator as Nagasawa already predicted in his textbook \cite{Nag21}.

Also, there are some known results suggesting non-trivial relationships between super-Brownian motions and Schr\"odinger operators as follows. 
Let $\Sigma\subset \R^d$ be a compact set. 

\begin{definition}
$\Sigma$ is a \textit{removable set} for $\Delta$ in $L^2(\R^d)$ if and only if $\left.\Delta\right|_{C_c^\infty(\R^d\backslash \Sigma)}$ is essentially self-adjoint in $L^2(\R^d)$.
\end{definition}

\begin{definition}
$\Sigma$ is a \textit{polar set} for super-Brownian motion $\{X_t\}_{t\geq 0}$ if and only if $\mathtt{P}_{\delta_x}\left(\mathrm{supp}(X_t)\cap \Sigma=\emptyset \textrm{ for all }t>0\right)=1$ for any $x\not\in \Sigma.$
\end{definition}

\begin{definition}
$\Sigma$ is a \textit{removable singularity} for all nonnegative solutions of equations $\Delta u=u(x)^2$ if and only if $u\geq 0$ and \begin{align*}
\Delta u(x)=u(x)^2,\quad x\in \R^d\backslash \Sigma,
\end{align*}
then $u=0$.
\end{definition}

\begin{definition}
The Bessel capacity of a compact set $\Sigma\subset \R^d$ with index $(\gamma,p)$ for $p>1$ is defined by \begin{align*}
\mathrm{Cap}_{\gamma,p}(\Sigma)=\sup\left\{\nu(\Sigma):\int_{\R^d}\dd x\left(\int_{\Sigma}G_\gamma(x,y)\nu(\dd y)\right)^{q}\leq 1\right\},
\end{align*} 
where $q=\frac{p}{p-1}$ and \begin{align*}
G_\gamma(x,y)=a_\gamma\int_0^\infty t^{\frac{\gamma-d}{2}}e^{-\frac{|x-y|^2}{2t}}\frac{e^{-\frac{t}{2}}}{t}\dd t
\end{align*}
($a_\gamma>0$ is a constant) and the supremum is taken over {all finite measures}.
\end{definition}

Then, it is known that for any compact set $\Sigma\subset \R^d$, the followings are equivalent:\begin{enumerate}[label=(\arabic*)]
\item \label{item:Remset} $\Sigma$ is a removable set for $\Delta$.
\item \label{item:Polar} $\Sigma$ is a polar set for super-Brownian motion.
\item \label{item:RemSin} $\Sigma$  is a removable singularity for all  nonnegative  solutions of equations $\Delta u=u(x)^2$.
\item \label{item:Bess} $\mathrm{Cap}_{2,2}(\Sigma)=0$.
\end{enumerate}

The equivalence \ref{item:Remset}$\Leftrightarrow$\ref{item:Bess} is proved in \cite{HKM17}, \ref{item:Polar}$\Leftrightarrow$\ref{item:RemSin} is proved in \cite{Dyn91}, and \ref{item:RemSin}$\Leftrightarrow$\ref{item:Bess} is proved in \cite{BP84}. Le Gall proved the one-to-one correspondence between all solutions of $\Delta u=u^2$ and the finite measures on $\Sigma$ that do not charge sets of zero capacity if $\Sigma$ is not a removable singularity in \cite{LeG95}.

In particular, \cite{HKM17} pointed out the characterization of the removability of $\Sigma$ via \textit{two-parameter additive Brownian motion} and \textit{$\R^d$-valued two parameter Brownian sheet}. Since \ref{item:Remset}$\Leftrightarrow$\ref{item:Polar} tells the removability of $\Sigma$ can be characterized in terms of super-Brownian motion,  it is plausible to generalize the result by Fleischman-Mueller \cite{FM02} as follows:

\begin{conjecture}
Let $\Sigma\subset \R^d$ be a compact set which is not a removable set for $\Delta$ in $L^2(\R^d)$. Then, there exist a self-adjoint extension  of $\Delta|_{C_c^\infty(\R^d\backslash \Sigma)}$ in $L^2(\R^d)$, $\widetilde{\Delta}$, and  a measure valued process $\left\{\mathcal{X}_t\right\}_{t\geq 0}$ such that  \begin{align*}
\mathtt{E}_{\delta_x}[\mathcal{X}_t(u_0)]=\widetilde{u}(t,x)
\end{align*}
for any $t>0$ and any $x\in \R^3\backslash\Sigma$, where $\widetilde{u}$ is the solution to \begin{align*}
\partial_t u=\frac{1}{2}\widetilde{\Delta}u,\quad u(0,x)=u_0(x).
\end{align*}
\end{conjecture}

\appendix 
\section{$V$-Uniform  Ergodicity}\label{app:Vuni}

Here, we collect some terminologies and results on the ergodic theory of Markov chains from \cite{Bax05,MT09}.  As a result, we can obtain the $V$-uniform ergodicity of $\{M_n^{1}\}_{n\geq 0}$ (Lemma \ref{lem:VuniGeoMe}).

We consider a Markov chain $X=\{X_n\}_{n\geq 0}$ on a state space $(\R^d,\mathcal{B}(\R^d)$, with transition probabilities
 \begin{align*}
P^n(x,A)=P(X_n\in A|X_0=x),\quad x\in \R^d,A\in \mathcal{B}(\R^d),n\geq 0.
\end{align*}
We assume that for each $x\in \R^d$ and $n\geq 0$, $P^n(x,\cdot)$ is a measure on  $(\R^d,\mathcal{B}(\R^d))$, and for each $A\in \mathcal{B}(\R^d)$  $n\geq 0$, $P^n(\cdot,A)$ is a measurable function on $\R^d$.  In this section, we focus on the case where $X$ has a unique invariant distribution $\pi$, i.e. there exists a probability measure $\pi$ on $\mathcal{B}(\R^d)$ such that for any $A\in\mathcal{B}(\R^d)$ \begin{align*}
\int_{\R^d} P^1(x,A) \pi(\dd x)=\pi(A).
\end{align*} 
Then, it is known from \cite{JJ67}(see \cite{NT78}) that \begin{align}
\|P^n(x,\cdot)-\pi(\cdot)\|_{\mathrm{TV}}\to 0,\quad  \pi\text{-a.a.$x$,} \label{eq:TVconv}
\end{align}
as $n\to \infty$, where $\|\cdot\|_{\mathrm{TV}}$ denotes the total variation of the measure.  In particular, we are interested in the rate of this convergence. 

The Markov chain $X$ is called \textit{$V$-uniform ergodic} if there exist a positive, Borel measurable function $V\geq 1$ with $\int V(x)\dd \pi<\infty$, $M>0$, and $q\in (0,1)$ such that  \begin{align*}
\sup_{|f|\leq V}\|P^n(x,f)-\pi(f)\|= MV(x)\rho^n 
\end{align*}
for all $x\in X$, where the supremum is taken over all measurable $f:\R^d\to \R$ satisfying $|f(x)|\leq V(x)$ for all $x\in \R^d$. In particular,  $V$-uniform ergodicity implies \eqref{eq:TVconv}.

In \cite{Bax05}, a sufficient condition for $V$-uniform ergodicity for the Markov chain $X$ is given.

\begin{theorem}\label{thm:GeoErgcond}(\cite[Theorem 1.1]{Bax05})
Suppose $X$ satisfies the following conditions:
\begin{enumerate}[label=(\roman*)]
\item \textit{Minorization condition}. There exists $B\in \mathcal{B}(\R^d)$, $\tilde{\beta}>0$ and a probability measure $\nu$ on $(\R^d,\mathcal{B}(\R^d)$ such that \begin{align*}
P(x,A)\geq \tilde{\beta}\nu(A)
\end{align*} 
for any $x\in B$ and $A\in \mathcal{B}(\R^d)$.
\item\label{item:Drift} \textit{Drift condition}. There exist a measurable function $V:\R^d\to [1,\infty)$ and  constants $\lambda<1$ and $K<\infty$ satisfying \begin{align*}
\int_{\R^d}V(y) P(x,\dd y)\leq \begin{cases}
\lambda V(x),\quad &x\not\in B\\
K,\quad &x\in B.
\end{cases}
\end{align*}
\item \textit{Strong aperiodicity condition}. There exists $\beta>0$ such that $\tilde{\beta}\nu(C)\geq \b$. 
\end{enumerate}
Then, $X$ is $V$-uniform ergodic, where $V$ is a measurable function given in \ref{item:Drift}.

\end{theorem}

We apply Theorem \ref{thm:GeoErgcond} to the Markov chain $\{M_n^\e\}_{n\geq 0}$.  
\begin{lemma}\label{lem:VuniGeoMe}
$\{M_n^1\}_{n\geq 0}$ is $V$-uniform integrable with $V(x)=\max\{|x|,1\}$.
\end{lemma}
\begin{proof}
It is enough to verify the three conditions in Theorem \ref{thm:GeoErgcond}.

First, we show that there exist $R>10$,  $B=B_R(0):=\{x\in \R^3;|x|\geq R\}$  such that Drift condition condition in Theorem \ref{thm:GeoErgcond} is satisfied.

By definition, we have \begin{align*}
\int_{\R^3}V(y) \mathtt{P}_x(\dd y)&=|x|\int_{|y|\leq 1} \frac{2(\pi)^{\frac{3}{2}}p_1(y)}{2\pi|x-y|\int_0^{\frac{|x|^2}{2}}u^{-\frac{1}{2}}e^{-u}\dd u}\dd y\\
&+|x|\int_{|y|> 1,|x-y|\leq \frac{R}{2}} \frac{2(\pi)^{\frac{3}{2}}|y|p_1(y)}{2\pi|x-y|\int_0^{\frac{|x|^2}{2}}u^{-\frac{1}{2}}e^{-u}\dd u}\dd y \\
&+|x|\int_{|y|> 1,|x-y|> \frac{R}{2}} \frac{2(\pi)^{\frac{3}{2}}|y|p_1(y)}{2\pi|x-y|\int_0^{\frac{|x|^2}{2}}u^{-\frac{1}{2}}e^{-u}\dd u}\dd y\\
&\leq \frac{C|x|}{(R-1)\int_0^{\frac{R^2}{2}}u^{-\frac{1}{2}}e^{-u}\dd u}\\
&+|x|\int_{|y|> 1} \frac{C|y|p_1(y)}{R\int_0^{\frac{R^2}{2}}u^{-\frac{1}{2}}e^{-u}\dd u}\dd y\\
&+|x|\int_{|x-y|\leq \frac{R}{2}} \frac{CR e^{-\frac{R^2}{8}}}{|x-y|\int_0^{\frac{R^2}{2}}u^{-\frac{1}{2}}e^{-u}\dd u}\dd y\\
&=C_R|x|.
\end{align*}
Taking $R>10$ large enough, it is easy to see that $C_R\to 0$. Thus, we fix $R>10$ such that  $C_R<1$. 

Since  $\int_{\R^3}V(y) \mathtt{P}_x(\dd y)$ is continuous in $x\in \R^3$, there exists a constant $K>0$ such that $\int_{\R^3}V(y) \mathtt{P}_x(\dd y)\leq K$ for all $|x|\leq R$.

For fixed $R\geq 0$, we define a probability measure on $\nu_R$ by \begin{align*}
\frac{\nu_R(\dd y)}{\dd y}=\frac{1}{Z_R}\int_{|x|\leq R}\dd x p_1(x)H(0,x-y)p_1(y),
\end{align*}
where $Z_R$ is the constant such that $\nu_R$ becomes a probability measure.

Then, it is easy to see that for any $A\in \mathcal{B}(\R^3)$, \begin{align*}
&\mathtt{P}_x\left(A\right)=\frac{1}{H(1,x)}\int_{A}\dd x H(0,x-y)p_1(y)\dd y\\
& \nu_R(A)=\frac{1}{Z_R}\int_{|x|\leq R}\int_{A}\dd x \dd y p_1(x)H(0,x-y)p_1(y)
\end{align*}
Since $H(1,x)$ and $p_1(x)$ are continuous and bounded away from $0$ and $\infty$  on $|x|\leq R$, we can find a $\tilde{\b}>0$ such that \begin{align*}
\mathtt{P}_x(A)\geq \tilde{\beta}\nu_R(A)
\end{align*}
 for any $x\in B_R$ and $A\in \mathcal{B}(\R^3)$.
 
We omit the proof of strong aperiodicity condition.

\begin{rem}
Modifying the proof, we can show that $V$-uniform ergodicity for $V$ with power growth functions, exponential growth functions, e.t.c. 
\end{rem}

\end{proof}

\textbf{Acknowledgments}
This work was supported by JSPS KAKENHI Grant Numbers JP22H01128, JP18K13423, JP22K03351. The author thanks Dr. Barkat Mian for discussions on the topic of this article. 





\end{document}